\documentclass[11pt]{article} 

\usepackage[utf8]{inputenc} 
\usepackage{float}
\usepackage{geometry} 
\usepackage{amsmath}
\allowdisplaybreaks
\usepackage{blkarray}
\usepackage{multirow}
\usepackage{graphicx} 
\usepackage{setspace}
\usepackage[toc,page]{appendix}
\usepackage{color}
\usepackage{booktabs} 
\usepackage{array} 
\usepackage{paralist} 
\usepackage{amsfonts}
\usepackage{ragged2e}
\usepackage{verbatim} 
\usepackage{subfig} 
\usepackage{amsthm}
\usepackage{color}
\usepackage{fancyhdr} 
\usepackage{sectsty}
\allsectionsfont{\sffamily\mdseries\upshape} 
\usepackage[mathscr]{euscript}
\usepackage[nottoc,notlof,notlot]{tocbibind} 
\usepackage[titles,subfigure]{tocloft} 
\usepackage{algorithm}
\usepackage{algpseudocode}
\usepackage{algorithmicx}
\usepackage{amssymb}
\usepackage{arydshln}

\usepackage{graphicx}

\algdef{SE}[SUBALG]{Indent}{EndIndent}{}{\algorithmicend\ }%
\algtext*{Indent}
\algtext*{EndIndent}
    \newtheorem{theorem}{Theorem}
    \newtheorem{lemma}{Lemma}
    \newtheorem{proposition}{Proposition}

\newtheorem{definition}{Definition}

    \newcounter{example}

\title{Worst-Case Analysis for a Leader-follower Partially Observable Stochastic Game}
\author{Yanling Chang\\
Department of Engineering Technology \& Industrial Distribution \\
 Department of Industrial Systems and Engineering,\\
  Texas A\&M University, College Station, TX 77843\\
   yanling.chang@tamu.edu\\
   Chelsea C. White III\\
	H. Milton Stewart School of Industrial \& Systems Engineering,\\ Georgia Institute of Technology, Atlanta, GA 30318,\\
	cw196@gatech.edu}
\date{}
\begin{document}
\maketitle
\doublespacing
\begin{abstract}

Partially observable stochastic games provide a rich mathematical paradigm for modeling multi-agent dynamic decision making under uncertainty and partial information. However, they generally do not admit closed-form solutions and are notoriously difficult to solve. Also, in reality, each agent often does not have complete knowledge of the other agent. This paper studies a leader-follower partially observable stochastic game where the leader has little knowledge of the adversarial follower's reward structure, level of rationality, and process for gathering and transmitting data relevant for decision making. We introduce the worst-case analysis to the partially observable stochastic game to cope with this lack of knowledge and determine the best worst-case value function of the leader. The resulting problem from the leader's perspective has a simple sufficient statistic; however, different from a classical partially observable Markov decision process, the value function of the resulting problem may not be convex. We design a viable and computationally attractive solution procedure for computing a lower bound of the leader's value function as well as its associated control policy in the finite planning horizon. We illustrate the use of the proposed approach in a liquid egg production security problem.
\end{abstract}
Keywords: worst-case analysis; partially observable Markov decision process; partially observable stochastic game.

\section{Introduction}
Stochastic games are classical dynamic game models, where the state of the system evolves on the basis of the current state and actions taken by all agents. However, these models do not consider the fact that in reality, an agent in a multi-agent scenario is likely to have only partial and noise-corrupted data about other agents. Partially observable stochastic games (POSGs) generalize stochastic games by taking into consideration this fact, and provide a rich normative framework for multiple intelligent agents with distinct objectives to dynamically control the operation of a system under uncertainty and partial information. Unfortunately, while game theorists have heavily studied Bayesian games and stochastic games, the literature on POSGs is relatively sparse. POSGs generally do not admit closed-form solutions and suffer from significant computational challenges. Tractable algorithms for computing control policies for general-sum POSGs are still rare, and both the Artificial Intelligence (AI) and Operations Research (OR) communities have focused on POSGs with special structures including zero-sum POSGs (Ghosh et al. 2004; Saha 2014) and common-payoff POSGs (Seuken and Zilberstein 2005; Oliehoek and Amato 2016) over the past decades.   

A leader-follower POSG introduced in Chang et al. (2015a, 2015b) is a new general-sum POSG, where each agent knows its own state but only has possibly inaccurate/incomplete observation of the other agent's state. At the beginning of the game, the leader selects its policy first, and then the follower determines its best response policy, with the complete knowledge of the leader's policy (e.g., adversaries spend time in learning defender's policies (\textit{Information Operations}, 2014)). At each stage, the selected policy pair determines actions for each agent simultaneously, and then the system transitions to a new state and each agent receives a new observation. Each agent's policy selects action to achieve its objective, based on the entire history of its own current and past observations, states, and actions.  

The (single-period) leader-follower game, also called the Stackelberg game, has wide applications and has been successfully implemented in many real security problems. A main focus of these applications is to provide decision support to the defender (commonly modeled as the leader) who is protecting a set of critical nodes against adversaries. Such examples include the placement of checkpoints and canine units at Los Angeles International Airport (Jain et al. 2010) and the scheduling of patrols in the Port of Boston (An et al. 2014). However, in the presence of intelligent agents who can respond and adjust their actions in ever-changing environments over time, many of these problems also have a dynamic nature that cannot be fully addressed by single-period games (e.g., an adversary may choose another target if the current target is well protected). The leader-follower POSG is able to explicitly model the dynamic interaction between agents and determine dynamic defense policies for the leader that promptly adjust defensive resource allocation over time for protection, given all available real-time data. This development is consistent with the emerging ``Moving Target Defense" belief that nowadays it is impossible to attain perfect security and the aim should be to develop dynamic systems that are defensible rather than perfectly secure (Department of Homeland Security 2015).

Game-theoretic approaches commonly assume that each agent has complete knowledge of both agents' reward structures, dynamics, and data sensing and transmission systems and that each agent's objective is to maximize its expected reward criterion. However, these assumptions may not be realistic in many scenarios (Camerer 2011). The intent of an adversary can span a wide range of possibly unknown issues (Bier et al. 2007), perfect rationality is often an unlikely human behavior (March 1978), and action selection may be affected by a variety of issues, such as task complexity, the interplay between emotion and cognition, etc. (Conlisk 1996).

In this paper, we consider a leader-follower general-sum POSG where: (i) the objective of the adversarial follower is unknown to the leader; (ii) at each decision epoch, it is unclear what information (or observations) that the follower has collected and how the follower will make its decision; (iii) the follower can be irrational; and (iv) the state of each agent cannot be precisely observed by the other agent. This is realistic in many situations, for example, when the leader is facing with a new unknown adversary. The intent of this research is to determine the best worst-case value function and the corresponding control strategy for the leader under these circumstances.   

The worst-case analysis is a popular approach in single-period games to cope with the uncertainty of an agent's behavior towards others. For example, Gilboa and Schmeidler (1989), Lo(1996), and Marinacci(2000) examined normal form games where an agent's action is not exactly known by the other agent. Aghassi and Bertsimas(2006) and Yolmeh and Baykal-Gursoy(2017) used this approach to contend with the payoff uncertainty. Simchi-Levi and Wei(2015) and Caprara et~al.(2016) also determined performance benchmarks in one-shot security planning via the same approach. The benchmarks from the worst-case analysis were further used to reveal the value of improved understanding of the behavior of adversaries in single-period security applications (Nguyen et al. 2013). Kardes et~al.(2011) introduced robust optimization to stochastic games where reward structure and/or transition probabilities are uncertain. However, the related analysis has not been examined for general-sum POSGs. This research introduces the worst-case analysis to the leader-follower POSG to reduce this gap in the literature.

Contributions of this paper are summarized as follows. 
\begin{enumerate}[(i)]
	\item  We introduce the worst-case analysis to a leader-follower POSG to consider the case where the leader has little knowledge of the adversarial follower. This is a first step for further evaluating the value of an improved understanding of the adversarial follower in a dynamic, multi-agent partially observable stochastic system.
	
	\item (\textit{Theoretically}) We show that the POSG under the worst-case analysis is a single-agent dynamic decision-making problem based solely on information of the leader. This model is unique from the zero-sum POSG and has a sufficient statistic more computationally tractable than the existing sufficient statistics presented in the POSG literature. We investigate the structural properties of the leader's optimal value function and show that it may not be convex.   
	
	\item (\textit{Computationally}) The non-convex structural result limits the usefulness of existing partially observable Markov decision process (POMDP) algorithms in our problem. While the worst-case model can also be viewed as a POMDP with imprecise parameters under the criterion of ``maxmin", currently there are no general algorithms for these problems. In order to establish a bottom-line performance for the leader, we develop a novel backward recursive algorithm to construct a lower bound for the leader's finite-horizon value function and to determine its associated policy. We evaluate the quality of the solution and show that the lower bound is no worse than the value function associated with the \textit{second best} leader's action. We also show that this algorithm can approximately determine a lower bound for the infinite planning horizon problem.
	
	\item (\textit{Application wise}) We illustrate the use of the proposed model and solution procedure via a security problem, where the operations manager of a liquid egg production plant is protecting the facility against an adversary who intends to insert a biological toxin into the system. We test and validate the effectiveness of the developed dynamic defensive strategy using simulation. 
\end{enumerate}
This paper is organized as follows. Section \ref{lit} presents a literature review on stochastic games with incomplete information, POSGs, and POMDPs with imprecise parameters. In Section \ref{problemstatementsec}, we introduce the worst-case analysis model for the general-sum leader-follower POSG, where the objective of the leader is to maximize the expected total discounted reward under the worst-case scenario of the follower. Section \ref{StructuralResults} presents the structural results of the leader's optimal value function, followed by their computational implications. In Section \ref{LowerBoundSolutionApproach}, we propose a three-step solution procedure for constructing a lower bound of the leader's value function and its policy, consisting of the PURGE-step, the DOMINANCE-step, and the APPROXIMATION-step. We discuss each of these steps in detail in Sections \ref{PurgeOperationSection}-\ref{PiecewiseLinearConcaveApproximation}. Specifically, Section \ref{PurgeOperationSection} utilizes an POMDP algorithm to eliminate redundant vectors in constructing the value function for a given leader action; Section \ref{DominanceOperationSection} presents a geometric approach and a mixed integer program to determine the optimal value function; and Section \ref{PiecewiseLinearConcaveApproximation} approximates the resulting value function by a piecewise linear and concave function. The approximation solution is used in the next iteration of the recursive algorithm. We also analyze the error bound of this approach and show our algorithm can be used to construct a lower bound for the leader's infinite-horizon optimal value function. Section \ref{NumericalResults} illustrates the use of the developed approach in a security application. Finally, Section \ref{Conclusions} summarizes research results and discusses future research directions.

\section{Literature Review}
\label{lit}
This section briefly reviews stochastic games with incomplete information, partially observable stochastic games, and the POMDP with imprecise parameters.
\subsection{Stochastic Games with Incomplete Information}
Stochastic games were first developed in Shapley(1953). Afterwards, many extensions were developed to consider the \textit{incomplete information} case where the reward structure and/or transition probabilities are imprecise. For instance, stochastic games with a single non-absorbing state where the payoff follows a given probability distribution were examined by Sorin (1984,1985). Najim et al.(2001) studied the optimization of the limiting average payoff of a zero-sum stochastic game with unknown transition probabilities and average payoffs. Luque-Vasquez and Minjarez-Sosa (2013) and Minjarez-Sosa and Vega-Amaya(2009) considered a zero-sum stochastic game where the payoffs are possibly unbounded. Cheng et al.(2016) developed two approximation methods to solve a two-person zero-sum stochastic game where the payoff matrix entries are independent and normally distributed. Kardes et al.(2011), Kardes(2014) and Rosenberg et al.(2004) analyzed equilibrium points for stochastic games with uncertain payoffs and/or transition probabilities. In all these studies, the state of the system is perfectly observable to each agent. 

\subsection{Partially Observable Stochastic Games} 
POSGs are stochastic games with \textit{imperfect information} where the state of the system is partially observed. A POSG can be transformed to a normal-form game and theoretically solved by iterated elimination of dominated strategies; however, this representation is often too large and not computationally feasible (Hansen et al. 2004). Also, it is impossible to transform POSGs into completely observable stochastic games over belief states, analogous to how a POMDP is solved by transforming it into a MDP over belief states (Hansen et al. 2004). At each stage, each agent will receive a unique observation, leading to various different,  possibly conflicting belief states (Emery-Montemerlo et al. 2004). Moreover, in a multi-agent system, each agent must also consider other agents' beliefs to reason their actions. Even worse, each agent must reason about the beliefs that other agents hold about each other's beliefs, leading to infinitely nested beliefs. As a result, POSGs suffer significant computational challenges: finding an optimal solution or even computing solutions with absolutely bounded error to common-payoff POSGs is $NEXP$-complete (Bernstein et al. 2002; Rabinovich et al. 2012); Goldsmith and Mundhenk(2008) further showed that the complexity of competitive POSGs can rise to $NEXP^{NP}$ (problems are solvable by a $NEXP$ machine using an $NP$ set as an oracle). 

There are no known tractable algorithms for computing optimal (or even reasonable) policies for genera-sum POSGs. Nevertheless, the AI community has made tremendous progress for special classes of POSGs. For example, for two-agent \textit{zero-sum} POSGs, Ghosh et al.(2004) and Saha (2014) showed that the POSG can be transformed to a stochastic game if both agents share a single observation process. Wiggers et al.(2016) analyzed the structural properties of value functions for zero-sum POSGs. However, algorithms for determining optimal policies for zero-sum POSGs are still rare. For \textit{common-payoff} POSGs, also called decentralized POMDP, numerous exact and approximation algorithms have been developed (see reviews in Seuken and Zilberstein 2005; Oliehoek 2012; Oliehoek and Amato 2016). 

The leader-follower POSG in Chang et al. (2015a, 2015b) is a \textit{general-sum} POSG where the leader-follower relationship makes the POSG both theoretically and computationally attractive. Assuming each agent uses finite-memory policies and has complete knowledge of the game setup, there is a \textit{finite-dimensional} sufficient statistic (the belief state) that consolidates all available information history for each agent to make a decision. The belief state is not defined on the state space; rather, it is an agent's belief over the set of all possible finite information histories of the other agent. Consequently, an agent can infer via its belief state both the system's state and the other agent's action. The existence of the sufficient statistics further allows for transforming the general-sum POSG to special structured POMDPs. This model is applied to assess the value of misinformation and disinformation in modern warfare (Chang et al. 2019). While the existing work assumes that each agent has complete knowledge of the game setup (e.g., reward structure, perfect rationality), this paper examines the case where the leader has limited knowledge of the follower (i.e., POSG with incomplete information). 

Another related model is called Interactive POMDPs (Gmytrasiewicz and Doshi 2004). The Interactive POMDP (I-POMDP) is another generalization of POMDPs to multi-agent systems. In this framework, agents are described by a class of possible models, and these agent models are included in the definition of ``interactive states". An agent's belief over these interactive states is a sufficient statistic. However, an agent's belief is also a component of the other agents' models. Thus, these beliefs are infinitely nested, making the problem computationally complex. Existing approximation solution techniques include policy iteration (Sonu and Doshi 2012), point based value iteration (Doshi and Perez 2008), and interactive particle filtering (Doshi and Gmytraslewicz 2009).   

\subsection{Partially Observable Markov Decision Processes with Imprecise Parameters}  POMDPs with imprecise parameters were analyzed in Itoh and Nakamura (2007) under the notion of ``second-order beliefs" which are beliefs in the imprecisely specified transition probabilities and observation probabilities. The authors determined a set of optimal policies, each of which is optimal to at least one of such second-order beliefs. Saghafian (2018) examined the structural results of ambiguous POMDPs using ``$\alpha$-maxmin" expected utility. Within the ``maxmin" criterion, Osogami (2015) proved that the value function can still be convex using the Loomis' Minimax Theorem, given the uncertainty set of the POMDP model parameters is convex. Under the assumption of S-rectangularity, Rasouli and Saghafian (2018) defined a robust POMDP, developed dynamic programming equations for it, and showed a zero-sum POSG can be transformed to a robust POMDP. However, to our best knowledge, there are no general algorithms for POMDPs with imprecise parameters yet (under the criterion of ``maxmin"). We further remark that none of these assumptions necessarily holds in our problem (especially when facing with unknown adversaries in a security context) and we study \textit{general-sum} leader-follower POSGs.

\section{The Worst-Case Analysis Modeling} 
\label{problemstatementsec}
We consider a partially observable stochastic game with a leader ($L$) and a follower ($F$). At each stage, the leader first selects an action, and then the follower determines its action. Once the two actions are selected, the reward for each agent is realized and the system transitions to a new state according to a given probability. However, the leader's information on the adversarial follower is very limited. Specifically, 
\begin{enumerate}[(i)]
	\item \textit{Decision Horizon}: The decision epochs are $t=0,1,2,...,T$ where $T \leq \infty$. 
	\item \textit{State Space}: 	the leader's state $s_t^L\in S^L $ and the follower's state $s_t^F \in S^F$. The state spaces $S^L$ and $S^F$ are assumed to be finite. Each agent knows its own state but may have only inaccurate observations of the other agent's state.
	\item \textit{Action Space}: at each stage, the leader selects $a^L_t \in A^L$ and the follower selects its (true) action $\tilde{a}_t^F\in A^F$, assuming the action spaces $A^L$ and $A^F$ are finite. How the follower selects its action $\tilde{a}_t^F$ is unknown to the leader, and the follower is possibly irrational.  
	\item \textit{Observation Space}: at each stage, the leader receives noisy observation $z_t^L\in Z^L $ of the follower's state $s^F_t$, where the leader's observation space $Z^L$ is finite. What observations $z^F_t$ that the follower may collect is unclear to the leader. 
	\item 	\textit{The State Transition and Observation Probabilities}: The conditional probability for the leader $ P(z^L_{t+1}, s_{t+1}|s_t, a_t)$ is assumed given. The follower's observation probabilities are unknown. 	
	\item \textit{Reward Structure}: the leader's scalar reward is $r^L (s_t,a_t)$, given the state pair $s_t=(s_t^L,s_t^F), s_t^L \in S^L, s_t^F \in S^F$ and action pair $a_t=(a_t^L,a_t^F), a^L_t \in A^L, a^F_t \in A^F$;   the follower's reward $r^F (s_t,a_t)$ is unknown.
\end{enumerate}
The problem \textit{objective} is to determine the best worst-case value function for the leader and its associated policy, given the above assumptions.

Because the follower only chooses $\tilde{a}^F_t$ \textit{after} the selection of the leader's action and $\tilde{a}^F_t$ may not be observable, the worst-case analysis assumes that at each stage $t$, the leader
\begin{enumerate}[(i)]
	\item first \textit{predicts} the follower's worst response action $a_t^F \in A^F$ for each $a^L_t$ (note, $a_t^F \neq \tilde{a}_t^F$); 
	\item determines its best leader's action $a^L_t \in A^L$ (see Figure 1).  
\end{enumerate}
\begin{figure}[H]
	\centering
	\includegraphics[width=0.9\textwidth]{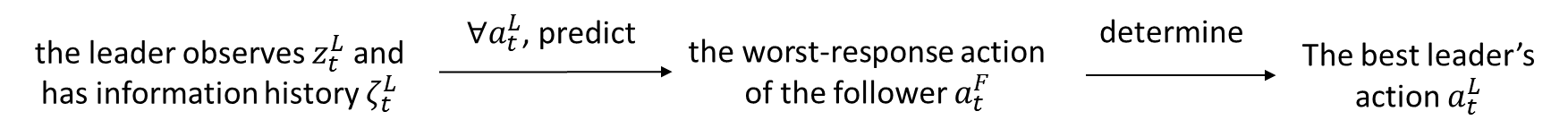}
	\caption{The decision making process of the leader. At each stage, the leader first predicts the follower's worst-response action $a^F_t (a^L_t)$ for each $a^L_t$ on the basis of $\zeta_t^L$, then determines its best leader action $a_t^L$. }
\end{figure}
\vspace{-10pt}
Thus, at time $t$, the leader's information history is $\zeta_t^L=\{s_t^L,...,s_0^L,z_t^L,...,z_1^L,a_{t-1},...,\\a_0,x_0^L\}$, where $x_0^L=\{P(s_0^F ),s_0^F \in S^F\}$ is the leader's prior probability mass vector over $S^F$ and $a_t=(a_t^L,a_t^F)$. Note that the ``predicted" follower's action $a_t^F$ is also a decision variable of the leader, same as $a_t^L$, and $\zeta^L_{t+1}=\{z^L_{t+1}, s^L_{t+1}, a_t, \zeta^L_t \}$. 

The criterion we consider $v_0^L(\zeta_0^L)$ is the expected total discounted reward accrued over horizon $T$. Namely, $v_0^L(\zeta_0^L) = E\{ \sum_{t=0}^{T} \beta^t r^L(s_t,a_t)|\zeta_0^L\}$ for the finite horizon case and $v_0^L(\zeta_0^L) = E\{ \sum_{t=0}^{\infty} \beta^t r^L(s_t,a_t)|\zeta_0^L\}$ for the infinite horizon case, where $E\{.|\zeta_0^L\}$ is the expectation operator conditioned on $\zeta_0^L$, and $\beta \geq 0$ is the discount factor. We assume $\beta<1$ for the infinite horizon case in order to ensure that $E\{ \sum_{t=0}^{\infty} \beta^t r^L(s_t,a_t)|\zeta_0^L\}$ is well defined. The problem objective is to determine a policy pair for the leader  $(\pi^{L,*}, \pi^{F,*}): \{\zeta_t^L\} \rightarrow A^L \times A^F$ such that 
\begin{equation}
v_0^{L,\pi^{L,*}, \pi^{F,*}}(\zeta_0^L) =\max_{\pi^{L} \in \Pi^L} \min_{\pi^{F} \in \Pi^F} E\{ \sum_{t=0}^{T} \beta^t r^L(s_t,a_t)|\zeta_0^L\}, \label{Obj}
\end{equation}
where $\Pi^k$ is the policy space of agent $k, k \in \{L,F\}$. 

The POSG under the worst-case scenario results in a \textit{single-agent} dynamic decision making problem: the leader determines both $a^L_t$ and the worst-case action $a^F_t$ on the basis of $\zeta^L_t$. Thus, this model is fundamentally different from a \textit{zero-sum} POSG (which itself is a challenging problem; see Section \ref{lit}). In zero-sum POSGs, agent $k$ makes its decision based on its own private information history $\zeta^k_t$. That is, assuming the follower is rational in the zero-sum POSG, the follower will select its (true) action $\tilde{a}_{t}^F$ on the basis of $\zeta^F_t$, whereas in the worst-case analysis, the leader has no knowledge of $\zeta^F_t$ and ``predicts" the worst-case action $a^F_t$ based on the leader's knowledge $\zeta_t^L, \zeta_t^L \neq \zeta_t^F$. Secondly, the follower may not be perfectly rational in our worst-case model.

Furthermore, knowing the exact information history that an adversary has is hard: $\zeta_t^F$ can be $\{s_t^F,..., s_0^F,z_t^F,...,z_1^F,\tilde{a}_{t-1}^F,...,\tilde{a}_0^F,x_0^F\}, x_0^F=\{P(s_0^L), s_0^L \in S^L\}$ or other possibly unknown forms. It is also unclear how the follower will utilize $\zeta_t^F$ to make its decision (e.g., the follower is myopic or irrational). The worst-case analysis requires no knowledge of (i) the follower's private information history $\zeta_t^F$, and (ii) how the follower selects its action $\tilde{a}_{t}^F$. Instead, the worst-case modeling analyzes from the leader's perspective; namely, at epoch $t$, the leader predicts the follower's worst response action and selects its best action all based on its own information history $\zeta_t^L$. As a result, Eq. \eqref{Obj} is the baseline performance of the leader; i.e., $\forall \pi^F\in \Pi^F,v_0^{L,\pi^{L,*},\pi^{F,*}} (\zeta_0^L ) \leq v_0^{L, \pi^{L,*},\pi^{F} } (\zeta_0^L)$ for each given $\zeta_0^L$, which is exactly the objective of the ``worst-case analysis" of this paper.   

\section{Structural Results} 
\label{StructuralResults}

Let $v_t^L(\zeta_t^L)$ be the maximal value of the worst-case expected total discounted reward to be accrued from epoch $t$ until $T$, given information history $\zeta_t^L$, then $v_t^L(\zeta_t^L)$ can be described recursively by
\begin{align}
v_t^L(\zeta_t^L) & = \max_{a_t^L \in A^L}\min_{a_t^F \in A^F} \Bigg\{\sum_{s_t^F} r^L(s_t,a_t)P(s_t^F|\zeta_t^L) \notag \\
&+ \beta \sum_{z^L_{t+1}}\sum_{s^L_{t+1}}P(z^L_{t+1},s^L_{t+1}|\zeta^L_t, a_t)v^L_{t+1}(\zeta^L_{t+1})  \Bigg\}. \label{POSGIteration}
\end{align}
Let $x_t^L=\{x_t^L(s^F_t), s^F_t \in S^F\}$, where $x_t^L(s^F_t)=P(s^F_t|\zeta_t^L)$. Thus, $x_t^L$ is a ``belief" array indicating the leader's inference about the follower's state $s^F_t$. Furthermore, the leader's belief process $x^L_t$ is a controlled Markov process as there is a function $\lambda^L$ depending on $z_{t+1}^L, s^L_{t+1}, s^L_t, x^L_t$ and $a_t=(a^L_t,a^F_t)$ such that
\begin{align}
x_{t+1}^L=\lambda^L(z^L_{t+1},s^L_{t+1},s^L_t, x_t^L,a_t)
\end{align} 
where 
\begin{align}
&\lambda^L(z^L_{t+1},s^L_{t+1},s^L_t,x_t^L,a_t)=P(s^F_{t+1}|\zeta^L_{t+1})=P(s^F_{t+1}|z^L_{t+1},s^L_{t+1},a_t,\zeta^L_t) \notag \\&= \frac{P(z^L_{t+1}, s_{t+1}|a_t, \zeta^L_t)}{P(z^L_{t+1}, s^L_{t+1}|a_t, \zeta^L_t)}= \frac{\sum_{s^F_t}P(z^L_{t+1},s_{t+1}|s_t,a_t)x^L(s^F_t)}{ \sum_{s^F_{t+1}}\sum_{s^F_t}P(z^L_{t+1},s_{t+1}|s_t,a_t)x^L(s^F_t)}, \label{POSGxt_1}
\end{align} 
Denote the bottom of Eq. \eqref{POSGxt_1} by
\begin{align}
\sigma^L(z^L_{t+1},s^L_{t+1},s^L_t,x_t^L,a_t) = P(z^L_{t+1}, s^L_{t+1}|a_t, \zeta^L_t), \label{POSGsigma}
\end{align}
assuming it is non-zero.
Let $V$ be the set of all bounded, real-valued functions on $S^L \times X^L$ having supremum norm $||v|| = \sup\{ |v(s^L,x^L)|: s^L \in S^L, x^L \in X^L\}$, where $X^L=\{\sum_{s^F} x^L(s^F)=1, x^L(s^F) \geq 0 \}$.  Then $(V, ||.||)$ is a Banach space. We say a real-valued function $f(s^L, x^L)$ for a fixed $s^L$ is piecewise linear on $X^L$ if there exists a set $\Gamma(s^L)$ depending on $s^L$, $|\Gamma(s^L)|<\infty$ such that: $\forall x^L \in X^L$, there is a $\gamma \in\Gamma(s^L) $ satisfying $f(s^L, x^L) = x^L\gamma$, where $x^L\gamma=\sum_{s^F}x^L(s^F)\gamma(s^F)$. 
Now, define the operator $H: V \rightarrow V$ as 
\begin{align}
&[Hv](s^L,x^L)\notag=\max_{a^L \in A^L}\min_{a^F \in A^F} \Bigg \{ x^Lr^L(s^L,a) \\ & +  \beta \sum_{z^{L'}}\sum_{s^{L'}} \sigma^L(z^{L'},s^{L'}, s^L,x^L,a) v(s^{L'}, \lambda^L(z^{L'},s^{L'}, s^L, x^L,a)) \Bigg \}, \label{POSGH}
\end{align}
where $x^Lr^L(s^L,a) = \sum_{s^F}x^L(s^F)r^L(s^L,s^F,a)$. 

\begin{proposition}
	\label{statistic}
	\begin{align*}
	v_t^L(\zeta_t^L)&=v_t^L(s^L_t, x^L_t)=[Hv^L_{t+1}](s^L_t,x^L_t).
	\end{align*}
	Thus, $v_t^L$ is dependent on $\zeta_t^L$ only through $(s^L_t, x_t^L)$, and $(s^L_t, x_t^L)$ is a sufficient statistic for $v_t^L$. Furthermore, if $v$ is piecewise linear in $x^L$, then $Hv$ is also piecewise linear in $x^L$. 
\end{proposition}

\proof{}
The first part follows by mathematical induction and the fact that both $\sigma^L(z^L_{t+1},s^L_{t+1},s^L_t,x_t^L,a_t)$ and $\lambda^L(z^L_{t+1},s^L_{t+1},s^L_t,x_t^L,a_t)$ are functions of $(s^L_t, x_t^L)$. 
For the second part, assume $v$ is piecewise linear. Equivalently, there exists a finite set $\Gamma(s^{L'})$ such that 
\begin{align*}
v(s^{L'},\lambda^L(z^{L'},s^{L'},s^L,x^L,a)) =  \lambda^L(z^{L'},s^{L'},s^L,x^L,a)\gamma^{l(z^{L'},s^{L'},s^L,x^L,a)}, 
\gamma \in \Gamma(s^{L'}),
\end{align*}

{\noindent where the function $l(z^{L'},s^{L'},s^L,x^L,a)$ defines the index of the $\gamma$ vector corresponding to $\lambda^L(z^{L'},s^{L'},s^L,x^L,a)$. 
	We say $\gamma^{'a} \in \Gamma^{'}(s^L,a)$, if $\gamma^{'a} $ is of the form} 
\begin{align*}
\gamma^{'a}(s^F) &=r^L(s,a)+\beta \sum_{z^{L'}}\sum_{s'}P(z^{L'},s' |s,a) \gamma(s^{F'})^{l(z^{L'},s^{L'},s^L,x^L,a)},
\end{align*}

{\noindent where $\gamma \in \Gamma(s^{L'})$. Note,
	$$[Hv](s^L,x^L)=\max_{a^L}\min_{a^F}\left\{ x^L\gamma^{'a}: \gamma^{'a} \in \Gamma'(s^L) \right\},$$
	where $\Gamma'(s^L)=\cup_{a \in A}\Gamma'(s^L,a)$. Hence, $Hv$ is piecewise linear in $x^L$ for $|A^L \times A^F| <\infty$. }   
\qedsymbol
\endproof

We remark that this sufficient statistic is more computationally tractable than other sufficient statistics presented in the POSG literature. In a multi-agent system, not only do agents have to infer the underlying system state, but also must consider other agents' beliefs in order to infer the other agents' actions.  As a result, the existing sufficient statistics are often complex and in high dimensions: they usually include a belief over the other agent's (finite) information history $\{P(\zeta^{F,\tau}_t|\zeta^L_t)\}$ where $\zeta^{F,\tau}_t = \{s_t^F,..., s_{t-\tau+1}^F,z_t^F,...,z_{t-\tau+1}^F,\tilde{a}_{t-1}^F,...,\tilde{a}_{t-\tau}^F\}$ under the finite-memory assumption (Chang et al. 2015a), or a belief over the other agent's complete model (Gmytrasiewicz and Doshi 2004). The worst-case modeling requires no such assumptions and the leader ``predicts" the follower's worst-case action based on its own knowledge, resulting in a much simpler sufficient statistic and a more computationally attractive problem. 

The worst-case analysis transforms a general-sum POSG to a single agent problem, and hence it is relatively simple compared to multi-agent problems. Further, Eq. \eqref{POSGH} is very similar to a POMDP with the newly defined $\lambda^L$ and $\sigma^L$, suggesting that solution procedures for the POMDP may be relevant to our problem. Unfortunately, the fact that the value function of a POMDP is both piecewise linearity and convexity forms the basis for existing POMDP algorithms (Sondik 1971; Cheng 1988; Pineau et al. 2003; Shani et al. 2013). While Proposition \ref{statistic} guarantees that the operator $H$ in Eq. \eqref{POSGH} preserves piecewise linearity, convexity will not be preserved, as illustrated by the following example. The non-convexity issue significantly limits the usefulness of solution procedures for the POMDP, and also makes our problem unique from the POMDP.  

Consider $|S^F|=|A^L|=2$ and let $v_T^L=0$. Fix a $s^L\in S^L$, then $v_{T-1}^L(s^L,x^L )=\max_{a^L}\min_{a^F} \left\{\sum_{s^F}x^L (s^F ) r^L (s^L,s^F,a^L,a^F ) \right\}=\max\{f_1(x^L), f_2(x^L)\}$ where $f_1(x^L) = \min \{4.6x_1^L + 7.6x_2^L, 8.2x_1^L + x_2^L \}$ (black line), $f_2(x^L)=\min \{1.8x_1^L+3.6x_2^L,0.6x_1^L+5.2x_2^L\}$ (blue line), and $(x_1^L,x_2^L)$ is the leader's belief vector over the follower's state, $x_1^L+x_2^L=1,x_1^L,x_2^L\geq 0$. Clearly, $v_{T-1}^L$ is non-convex in $x^L$ (red line). See Fig. \ref{PLnotC}. 
\begin{figure}[h]
	\centering
	\includegraphics[width=0.45\textwidth]{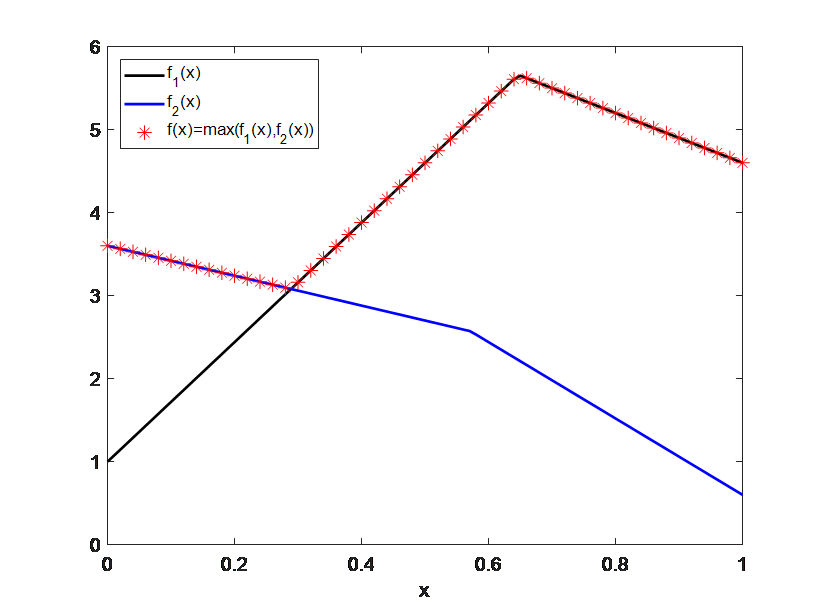}
	\caption{$Hv$ is piecewise linear but not necessarily concave or convex (because the ``maxmin" operation does not preserve convexity).}
	\label{PLnotC}
\end{figure}

For the infinite planning horizon, it is easy to show that 
\begin{proposition}
	$\forall 0\leq \beta < 1$, the operator $H$ is a contraction mapping on $V$ having modulus $\beta$.
	\label{ContractionMapping}
\end{proposition}
\proof{}
Let $u,v \in V$, fix $s^L, x^L$ and assume $Hu(s^L, x^L) \geq Hv(s^L, x^L)$, then	
\begin{align*}
& 0\leq [Hu](s^L,x^L)-[Hv](s^L,x^L) \leq \max_{a^L}\Bigg\{\min_{a^F \in A^F} \bigg [ x^Lr^L(s^L,a) \\&+  \beta \sum_{z^{L'}}\sum_{s^{L'}} \sigma^L(z^{L'},s^{L'}, s^L,x^L,a) u(s^{L'}, \lambda^L(z^{L'},s^{L'}, s^L,x^L,a)) \bigg ] 
\\
&-	\min_{a^F \in A^F} \bigg [ x^Lr^L(s^L,a) +  \beta \sum_{z^{L'}}\sum_{s^{L'}} \sigma^L(z^{L'},s^{L'}, s^L,x^L,a) v(s^{L'}, \lambda^L(z^{L'},s^{L'}, s^L,x^L,a)) \bigg]\Bigg\}
\end{align*}
For any $a^L$, let 
$a^{F,*} \in \arg\min \bigg [ x^Lr^L(s^L,a)  +  \beta \sum_{z^{L'}}\sum_{s^{L'}} \sigma^L(z^{L'},s^{L'}, s^L,x^L,a)\\* v(s^{L'}, \lambda^L(z^{L'},s^{L'}, s^L,x^L,a)) \bigg]$
then 
\begin{align*}
&\min_{a^F \in A^F} \bigg [ x^Lr^L(s^L,a) +  \beta \sum_{z^{L'}}\sum_{s^{L'}} \sigma^L(z^{L'},s^{L'}, s^L,x^L,a) u(s^{L'}, \lambda^L(z^{L'},s^{L'}, s^L,x^L,a)) \bigg ] \\&-	\min_{a^F \in A^F} \bigg [ x^Lr^L(s^L,a) +  \beta \sum_{z^{L'}}\sum_{s^{L'}} \sigma^L(z^{L'},s^{L'}, s^L,x^L,a) v(s^{L'}, \lambda^L(z^{L'},s^{L'}, s^L,x^L,a)) \bigg]\\
&\leq \bigg [ x^Lr^L(s^L,a^L, a^{F*}) +  \beta \sum_{z^{L'}}\sum_{s^{L'}} \sigma^L(z^{L'},s^{L'}, s^L,x^L,a^L, a^{F*}) u(s^{L'}, \lambda^L(z^{L'},s^{L'}, s^L,x^L,a^L, a^{F*})) \bigg ] 	\\&- \bigg [ x^Lr^L(s^L,a^L, a^{F*}) +  \beta \sum_{z^{L'}}\sum_{s^{L'}} \sigma^L(z^{L'},s^{L'}, s^L,x^L,a^L, a^{F*}) v(s^{L'}, \lambda^L(z^{L'},s^{L'}, s^L,x^L,a^L, a^{F*})) \bigg ] \\
&\leq \beta||u-v||, \forall a^L
\end{align*}	
Thus, $0\leq [Hu](s^L,x^L)-[Hv](s^L,x^L) \leq \beta ||u-v||$. Repeating this argument in the case that $0 \leq [Hv](s^L,x^L)-[Hu](s^L,x^L)$ shows that $|[Hu](s^L,x^L)-[Hv](s^L,x^L)| \leq \beta ||u-v||$. Taking the supremum over $s^L$ and $x^L$ gives $||Hu-Hv||\leq \beta||u-v||.$ 
\qedsymbol
\endproof

As a result, there is a unique fixed point $v^* \in V$ such that $Hv^* = v^*$. Let the sequence $\{v_n\}$ be defined as $v_{n+1} = Hv_{n}$, then $\lim_{n\rightarrow \infty}||v^*-v_n||=0$, given $\forall v_0 \in V$. Moreover, $v^*$ is continuous. However, it is not guaranteed that $v^*$ is convex or concave. 

We remark that the worst-case model can also be viewed as a POMDP with imprecise parameters, where the leader's reward $r^L$ and dynamics $P(z_{t+1}^L,s_{t+1} |s_t,a_t)$ are imprecise due to the unknown $a_t^F$. However, to our best knowledge, general algorithms do not exist for such problems with the ``maxmin" criterion (the existing literature focuses on imposing conditions to avoid the non-convexity issue; see Section \ref{lit}). In this paper, we propose a novel computationally attractive solution procedure without these conditions that could generate a lower bound for the finite-horizon value function with quantified error bound, and determine its associated control policy. Research on the infinite horizon case is a topic for future research. 

\section{Lower Bound Solution Approach and Policy Determination}
\label{LowerBoundSolutionApproach}
We motivate the lower bound solution approach by a rewrite of Eq. \eqref{POSGH} to two steps: the ``min" step and the ``max" step. Specifically, $\forall a^L \in A^L$, let the operator $H^{a^L}: V \rightarrow V$ be 
\begin{align*}
&[H^{a^L}v](s^L,x^L)=\min_{a^F \in A^F} \Bigg \{ x^Lr^L(s^L,a) \\ &+  \beta \sum_{z^{L'}}\sum_{s^{L'}} \sigma^L(z^{L'},s^{L'}, s^L,x^L,a) v(s^{L'}, \lambda^L(z^{L'},s^{L'},s^L, x^L,a)) \Bigg \},
\end{align*}
then, 
$$Hv = \max_{a^L \in A^L}[H^{a^L}v].$$

For the ``min" step, assume $\tilde{v}^L$ is a piecewise linear and concave approximation of $v^L$ satisfying $\tilde{v}^L \leq v^L$ (i.e., $\tilde{v}^L(s^L,x^L) \leq v^L(s^L,x^L), \forall s^L \in S^L, x^L \in X^L$). Namely, there is a finite set $\tilde{\Gamma}(s^L)$ such that $\tilde{v}^L(s^L,x^L) = \min \{x^L\gamma: \gamma \in \tilde{\Gamma}(s^L) \}$. Then, pick any $a^L \in A^L$,
\begin{align*}
&[H^{a^L}\tilde{v}^L](s^L,x^L)=\min_{a^F \in A^F} \Bigg \{ x^Lr^L(s^L,a) \\& +  \beta \sum_{z^{L'}}\sum_{s^{L'}} \sigma^L(z^{L'},s^{L'}, s^L,x^L,a) 
\times \min[ \lambda^L(z^{L'},s^{L'}, s^L,x^L,a)\gamma': \gamma' \in \tilde{\Gamma}(s^{L'}) ] \Bigg \}\\
&=\min_{a^F \in A^F} \Bigg \{ x^Lr^L(s^L,a)   +  \beta \sum_{z^{L'}}\sum_{s^{L'}} \min [\sum_{s^{F'}}\sum_{s^F}P(z^{L'},s'|s,a) x^L(s^F)\gamma'(s^{F'}): \gamma' \in \tilde{\Gamma}(s^{L'}) ] \Bigg \}\\
&=\min\Bigg \{ x^L\gamma: \gamma \in G(s^L,a^L)  \Bigg \},
\end{align*}
where $\gamma \in G(s^L,a^L)$ if $\gamma = r^L(s,a)+\beta \sum_{z^{L'}}\sum_{s^{'}}P(z^{L'},s'|s,a)\gamma'(s^{F'})$. Thus, $H^{a^L}\tilde{v}^L$ is piecewise linear and concave, and the ``min" step is a POMDP step. 

For the ``max" step, since $Hv = \max_{a^L \in A^L}[H^{a^L}v]$, we have the following: 
\begin{theorem}	\label{doubleindex}
	There is a finite set of arrays $\Gamma(s^L) = \{\gamma^{k_1,k_2}\}_{k_1, k_2 \geq 0}$ only depends on $s^L$ such that 
	$$[H\tilde{v}^L](s^L,x^L) = \max_{k_1}\min_{k_2}\{x^L\gamma^{k_1,k_2}: \gamma^{k_1,k_2} \in \Gamma(s^L) \}.$$
\end{theorem}
\proof{}
See the paragraphs above the Theorem. \qedsymbol
\endproof
\textbf{Policy Determination:} Each element of $\Gamma(s^L)$ is associated with a pair of action $(a^L, a^F)$. Denote $\Gamma(s^L,k_1)=\{\gamma^{k_1',k_2'}: \gamma^{k_1',k_2'} \in \Gamma(s^L), k_1' =k_1\}$. Each \textit{set} $\Gamma(s^L, k_1)$ corresponds to a leader's action $a^L$, and each \textit{vector} $\gamma \in \Gamma(s^L,k_1)$ is associated with a follower's worst-case action $a^F$. The best worst-case policy for $H\tilde{v}$ can thus be determined by the following steps: 
\begin{enumerate}[(i)]
	\item determine $(s^L,x^L)$;
	\item for each $k_1$, find $\gamma^{k_1,*}$ in $\arg\min\{x^L\gamma: \gamma \in \Gamma(s^L,k_1)\}$, and let $\Theta(s^L)=\cup_{k_1} \{\gamma^{k_1,*} \}$;
	\item determine $\gamma^{*,*}$ in $\arg\max\{x^L \gamma: \gamma \in \Theta(s^L)\}$;
	\item select the action pair associated with $\gamma^{*,*}$.
\end{enumerate}

In a nutshell, at time t, if there is a $\tilde{v}^L_{t+1}$ being a piecewise linear concave approximation of $v_{t+1}^L$, $ \tilde{v}^L_{t+1} \leq v_{t+1}^L$, we can determine $H^{a^L} \tilde{v}^L_{t+1}$ and $H\tilde{v}^L_{t+1}$ by POMDP algorithms and Theorem \ref{doubleindex}, respectively. Moreover, 
\begin{proposition}
	If $\tilde{v}^L_{t+1} \leq v_{t+1}^L$, then $\bar{v}_t^L=H\tilde{v}^L_{t+1} \leq v^L_t=Hv_{t+1}^L$.
	\label{perserveOrder}	
\end{proposition}
\proof{}
This is due to the definition of $H$ and $\sigma^L(z^{L'},s^{L'},s^L,x^L,a) \geq 0$. 
\qedsymbol
\endproof

As $\bar{v}_t^L$ is again non-convex, we then need to find another $\tilde{v}^L_t$ to best approximate $\bar{v}_t^L$ and repeat the process for time $t-1$. Fig. \ref{ThreeStep} presents a three-step procedure to implement this idea for the finite planning horizon problem.

\begin{figure}[h]
	\begin{algorithmic}\linespread{0.83}\selectfont		
		\State\parbox[t]{\dimexpr\linewidth-\algorithmicindent}{	\hangindent=30pt{\texttt{PURGE-step.} Determine $\bar{v}^{a^L}_t(s^L,x^L)=[H^{a^L}\tilde{v}_{t+1}^L](s^L,x^L) = \min\{x^L\gamma: \gamma \in G(s^L, a^L) \}$, for each leader's action $a^L$. POMDP techniques can be employed to efficiently eliminate redundant $\gamma$-vectors in $G(s^L,a^L)$ by the PURGE operator (the ``min" step).\\} }
		\State\parbox[t]{\dimexpr\linewidth-\algorithmicindent}{	\hangindent=30pt{\texttt{DOMINANCE-step.} Determine $\bar{v}_t^L(s^L,x^L)=[H\tilde{v}_{t+1}^L](s^L,x^L) = \max_{k_1}\min_{k_2}\{x^L\gamma^{k_1,k_2}: \gamma^{k_1,k_2} \in \Gamma_{t}(s^L) \}$. While $\Gamma_{t}(s^L)$ can be $\cup_{a^L \in A^L}G(s^L,a^L)$, we seek to effectively remove redundant sets $G(s^L,a^L)$ in $\Gamma_{t}(s^L)$ quickly by the DOMINANCE operator via computational geometry and mixed integer programs (the ``max" step).\\}}
		\State\parbox[t]{\dimexpr\linewidth-\algorithmicindent}{	\hangindent=30pt{\texttt{APPROXIMATION-step.} Determine $\tilde{\Gamma}_t(s^L)$ where $\tilde{v}_t^L(s^L,x^L) = \min\{x^L\gamma: \gamma \in \tilde{\Gamma}_t(s^L) \}$ is the best piecewise linear concave approximation of $\bar{v}_t^L$ satisfying $\tilde{v}_t^L \leq \bar{v}_t^L$.\\}}
	\end{algorithmic}
	\caption{The three steps for the lower bound solution approach}\label{ThreeStep}
\end{figure}
Specifically,
given $\tilde{v}_{t+1}^L$ being a piecewise linear concave approximation of $v^L_{t+1}$, $\bar{v}^{a^L}_t(s^L,x^L) =[H^{a^L}\tilde{v}_{t+1}](s^L,x^L)= \min\{x^L\gamma: \gamma \in G(s^L,a^L)\}$ is a POMDP step for $\forall a^L \in A^L$. As in all POMDP problems, the set $G(s^L,a^L)$ may contain many redundant $\gamma$-vectors never used in determining $\bar{v}^{a^L}_t$. The PURGE-\textit{step} is to remove all redundant $\gamma$-vectors in each $G(s^L,a^L)$, which can be accomplished by the PURGE operator in the POMDP literature. According to Theorem \ref{doubleindex} and $H\tilde{v}_{t+1}=\max_{a^L \in A^L}[H^{a^L}\tilde{v}_{t+1}]$, we could set $\Gamma_t(s^L) =G(s^L)$ where $G(s^L)=\cup_{a^L \in A^L}G(s^L,a^L)$; however, the resulting set can be very large. The  DOMINANCE-\textit{step} is to determine the set $\Gamma_t(s^L)$ that has the smallest cardinality and satisfies $\max_{k_1}\min_{k_2}\{x^L\gamma^{k_1,k_2}: \gamma \in \Gamma_t(s^L) \} = \max_{a^L}\min_{k_2}\{x^L\gamma^{a^L,k_2}: \gamma^{a^L,k_2} \in G(s^L) \}$ for computational advantage. The DOMINANCE operator is designed via computational geometry and mixed integer programs (MIPs). Because the resulting $\bar{v}_{t}^L=H\tilde{v}_{t+1}^L$ is again non-convex, the APPROXIMATION-\textit{step} approximates $\bar{v}_{t}^L$ with quantified error by a piecewise linear and concave function $\tilde{v}_t^L$ satisfying $\tilde{v}_t^L \leq \bar{v}_{t}^L$ for the next iteration.

Performing all required operations and approximation, we have developed a backward recursive algorithm for determining a lower bound of the leader's best worst-case value function and its associated control policy for the finite-horizon POSG. The pseudocode of the entire procedure is summarized in Algorithm \ref{overall}. The rest of the paper presents each step in more detail. 

\begin{algorithm}[h]
	\caption{Entire Algorithm for a Finite Horizon Partially Observable Stochastic Game}
	\begin{algorithmic}	\linespread{0.83}\selectfont	
		\State{\texttt{Set $\Gamma_T(s^L) = \emptyset$, $\tilde{\Gamma}_T(s^L) = \emptyset$, $\forall s^L \in S^L$, and $t=T-1$. }}
		\While{($t\geq 0$)} 
		\For{\texttt{each $s^L \in S^L$}}
		\State{PURGE-\textit{step:}}
		\Indent
		\For{\texttt{each $a^L \in A^L$}}
		\State\parbox[t]{\dimexpr 0.9\linewidth-\algorithmicindent}{ \texttt{Set
				$G(s^L,a^L)= \cup_{a^F \in A^F}\Bigg\{r^L(s,a) +\beta\sum_{s'}\sum_{z^L}P(s',z^L|s,a)\gamma(s^{F'}): \gamma \in \tilde{\Gamma}_{t+1}(s^{L'}) \Bigg \}.$ }}
		\State\parbox[t]{\dimexpr0.9\linewidth-\algorithmicindent}{\texttt{$G(s^L,a^L)$=PURGE($G(s^L,a^L)$) to remove redundant $\gamma$-vectors.}}
		\EndFor
		\EndIndent
		\State{DOMINANCE-\textit{step:}}
		\Indent
		\State{\texttt{Set $G(s^L) = \cup_{a^L \in A^L}G(s^L,a^L)$.}}
		\State\parbox[t]{\dimexpr0.88\linewidth-\algorithmicindent}{\hangindent=20pt{\texttt{Select the superset $\Gamma_t^c(s^L)$ out of $G(s^L)$ (Algorithm \ref{geometricalgorithm2}):\\Perform the pairwise dominance procedure on $G(s^L)$ to define the superset $\Gamma_t^c(s^L)$. The set $\Gamma_t^c(s^L)$ is a set of $G(s^L,a^L)$s such that $\forall G(s^L,a^L) \in \Gamma_t^c(s^L)$, there is no set $G(s^L,a^{L'})$, $a^{L'} \neq a^L$, dominating $G(s^L,a^L)$.\\}}\par}				 
		\State\parbox[t]{\dimexpr0.88\linewidth-\algorithmicindent}{\hangindent=20pt{\texttt{Select the set $\Gamma_t(s^L)$ out of $\Gamma_t^c(s^L)$ (Algorithm \ref{MIPDominancealgorithm}):\\Perform the jointly dominance procedure on $\Gamma_t^c(s^L)$ to further remove the subsets in $ \Gamma_t^c(s^L)$ that are strictly dominated by the set $\Gamma_t^c(s^L)$. Thus, $\bar{v}_t^L(s^L,x^L) =\\ \max_{k_1}\min_{k_2} \{x^L\gamma^{k_1,k_2}: \gamma^{k_1,k_2} \in \Gamma_t(s^L) \}$.}}\par}
		\EndIndent
		\State{APPROXIMATION-\textit{step:}}
		\Indent
		\State\parbox[t]{\dimexpr 0.88\linewidth-\algorithmicindent}{\texttt{Determine $\tilde{\Gamma}_t(s^L)$ where \\$\tilde{v}_t^L(s^L,x^L) = \min \{x^L\gamma: \gamma \in \tilde{\Gamma}_t(s^L) \}$ is the best piecewise linear concave approximation of $\bar{v}_t^L(s^L,x^L)$ and $\tilde{v}_t^L \leq \bar{v}_t^L$, and evaluate the approximation error $\epsilon_t(s^L)$ (Algorithm \ref{PLCA})}.}
		\EndIndent
		\EndFor
		\State Set $t=t-1$. 
		\EndWhile
	\end{algorithmic}
	\label{overall}
\end{algorithm}
\vspace{-3pt}
\section{PURGE Operation}
\label{PurgeOperationSection}
Assume a piecewise linear and concave function $\tilde{v}_{t+1}^L$ satisfying $\tilde{v}_{t+1}^L \leq v_{t+1}^L$ is given. That is, there is a set $\tilde{\Gamma}_{t+1}(s^L)$ for each $s^L \in S^L$ where $\tilde{v}_{t+1}^L(s^L,x^L) = \min \{x^L\gamma: \gamma \in \tilde{\Gamma}_{t+1}(s^L) \}$. For any $a^L \in A^L$, we have shown that   $\bar{v}^{a^L}_t(s^L,x^L)=[H^{a^L}\tilde{v}_{t+1}^L](s^L,x^L) = \min \{x^L\gamma: \gamma \in G(s^L,a^L) \}$ where 
\begin{align*}
G(s^L,a^L) = \cup_{a^F \in A^F} \Bigg\{r^L(s,a)+\beta \sum_{z^{L'}}\sum_{s^{'}}P(z^{L'},s'|s,a)\gamma'(s^{F'}): \gamma' \in \tilde{\Gamma}_{t+1}(s^{L'})\Bigg\}.
\end{align*}

A large number of $\gamma$-vectors could be generated in this step, however, only a small number of these vectors define $\bar{v}^{a^L}_t$. A $\gamma \in G(s^L, a^L)$ is \textit{redundant} if and only if for all $x^L \in X^L$, there is a $\gamma' \in G(s^L, a^L), \gamma' \neq \gamma$ such that $x^L\gamma' \leq x^L\gamma$; a $\gamma \in G(s^L, a^L)$ is referred to as a \textit{defining} vector for $\bar{v}^{a^L}_t$ if there is a belief point $x^* \in X^L$ such that $\bar{v}^{a^L}_t(s^L,x^*)=x^*\gamma$ and these belief points are called \textit{witness points} (Lin et al. 2004). The PURGE operators in Cassandra et al.(1997), Lin et al.(2004) (and many others) in the POMDP literature can be employed to efficiently remove redundant vectors from $G(s^L,a^L)$. We adopt the PURGE operator in Lin et al.(2004) in our illustrative example.

\setcounter{equation}{0}

\section{DOMINANCE Operation}
\label{DominanceOperationSection}
We now determine $\Gamma_t(s^L) \subseteq G(s^L)$ where $G(s^L)=\cup_{a^L \in A^L}G(s^L,a^L), \forall s^L \in S^L$ by extending the notion of redundancy of a $\gamma$-vector to a set. For a given $s^L \in S^L$, a set $G(s^L,a^L)$ is \textit{dominated} by $\Gamma_t(s^L)$ on $X^L$ if and only if $\forall x^L \in X^L$, there is always a set $G(s^L,a^{L'})$ in $\Gamma_t(s^L)$ such that $ \bar{v}^{a^{L}}_t(s^L,x^L) \leq \bar{v}^{a^{L'}}_t(s^L,x^L)$ where $\bar{v}_t^{a^L}=\min\{x^L\gamma: \gamma \in G(s^L,a^L) \}$; a set $G(s^L,a^{L})$ is referred to as \textit{supporting} if there is at least one belief point $x'\in X^L$ such that $\bar{v}_t^L(s^L,x')=\max_{k_1}\min_{k_2}\{x'\gamma^{k_1,k_2}: \gamma^{k_1,k_2} \in \Gamma(s^L) \}=\min\{x'\gamma: \gamma \in G(s^L,a^{L}) \}$. For example, both sets $G(s^L,a^L_1)$ and $G(s^L,a^L_4)$ in Fig. \ref{DominatedSupportingSets} are dominated sets, while sets $G(s^L,a^L_2)$ and $G(s^L,a^{L}_3)$ are supporting for $\bar{v}_t^L$.  We seek to remove all dominated sets in order to define $\Gamma_t$ efficiently. 
\begin{figure}[h]
	\centering
	\includegraphics[width=0.6\textwidth]{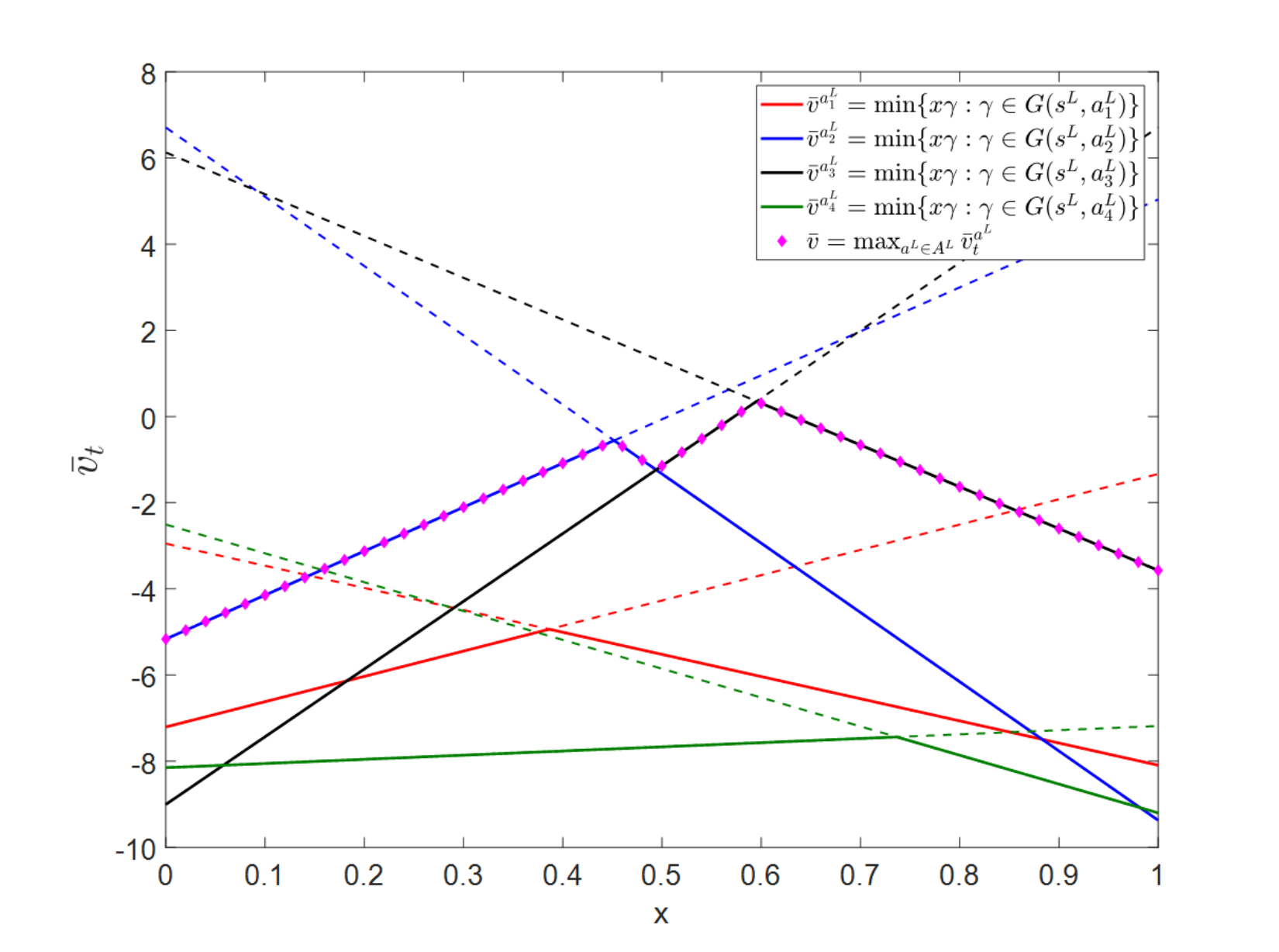}
	\caption{dominated and supporting sets}
	\label{DominatedSupportingSets}
\end{figure}

Let the DOMINANCE operator be that $\Gamma(s^L)=\mbox{DOMINANCE}(G(s^L))$ only contains supporting sets of $\bar{v}_t^L$. We present a two-step procedure for the DOMINANCE operator. We say a set $G(s^L,a^L)$ is \textit{pair-wise dominated} by a set $G(s^L,a^{L'})$ if and only if $\forall x^L\in X^L$, $\bar{v}^{a^{L}}(s^L,x^L) \leq \bar{v}^{a^{L'}}(s^L,x^L)$. For example, set $G(s^L,a^L_4)$ is pair-wise dominated by the set $G(s^L,a^L_1)$ in  Fig. \ref{DominatedSupportingSets}.

The first step is to build a superset $\Gamma_{t}^c(s^L) \supseteq \Gamma_{t}(s^L)$, where for any $G(s^L,a^{L})$ in $\Gamma_{t}^c(s^L)$, there is no $G(s^L,a^{L'})$ in $\Gamma_{t}^c(s^L)$ pair-wise dominating $G(s^L,a^{L})$. We show the pairwise dominance relationship between two sets in $G(s^L)$ can be determined efficiently by a sequence of linear programs (LPs), employing a dual relationship between hyperplanes and points. However, a set $G(s^L,a^{L})$ in $\Gamma_{t}^c(s^L)$ could still be a dominated set, such as $G(s^L,a^{L}_1)$ in Fig. \ref{DominatedSupportingSets}. The second step employs a MIP to further remove all dominated sets in $\Gamma_{t}^c(s^L)$. The first step is to substantially decrease the number and the size of MIPs encountered in the second step. 

\subsection{Determine the Superset $\Gamma_{t}^c(s^L) \supseteq \Gamma_{t}(s^L)$}
In computational geometry, the \textit{dual} of a hyperplane $p(u)=\sum_{i=1}^{n-1}p_iu_i +p_n$ in the primal $R^n$ space is the point $p^*=(p_1,...p_n) \in R^n$, and the dual of a point $p=(p_1,...p_n) \in R^n$ is a hyperplane $p^*(u)=\sum_{i=1}^{n-1}p_iu_i +p_n$. The \textit{lower envelope} of a given set of hyperplanes $\{p^k(u)=\sum_{i=1}^{n-1}p^k_iu_i +p^k_n\}_{k \in K}$ is the piecewise linear and concave function $\underbar{p}(u)=\min_{k \in K} \{\sum_{i=1}^{n-1}p^k_iu_i +p^k_n\}$, whereas the \textit{upper envelope} of the given set of hyperplanes is the piecewise linear and convex function $\bar{p}(u)=\max_{k \in K} \{\sum_{i=1}^{n-1}p^k_iu_i +p^k_n\}$. Each of the hyperplanes on the upper envelope in the primal space corresponds to a vertice of the \textit{upper convex hull} (with respect to the $p_n$-axis) in the dual space (de Berg et al. 2008; Zhang 2010). We need the following definitions in Zhang (2010) to proceed. 

Given a set $\Omega \in R^{|S^F|}$ of points, the \textit{convex hull} is the set $Co(\Omega) \equiv \{\sum_{j=1}\lambda_jw_j: \sum_{j=1}\lambda_j=1 \mbox{ and } w_j \in \Omega, \lambda_j \geq 0, \forall j \}$. The surface of the convex hull with negative outernormal directions, the \textit{negative convex hull} ($NCo$), is the set $NCo(\Omega) \equiv cl(\{w \in Co(\Omega): \exists x \in X^+, xw \leq x\gamma, \forall \gamma \in Co(\Omega) \})$, where $X^+ = \{x \in X \mbox{ and } x_i >0, \forall i \}$, and $cl(B)$ is the closure of $B$. Then,
\begin{lemma}
	Suppose that $\Omega =G(s^L,a^L) \in R^{|S^F|}$ is closed and bounded, hence, compacted. For any given $a^L$ and $s^L$, the piecewise linear and concave function $\bar{v}^{a^L}(s^L,x^L) = \min \{x^L\gamma: \gamma \in G(s^L,a^L) \}, x^L \in X^L $, is dual to the set $NCo(\Omega)$. Namely, for any $\hat{x}^L \in X^L$, there exists a $\hat{\gamma} \in NCo(\Omega)$ such that $\bar{v}^{a^L}(s^L,\hat{x}^L)= \hat{x}^L\hat{\gamma}$, and conversely, for any $\hat{\gamma} \in NCo(\Omega)$, there is a $\hat{x}^L \in X^L$ such that $\bar{v}^{a^L}(s^L,\hat{x}^L) = \hat{x}^L\hat{\gamma}$.
	\label{geometricLemma}
\end{lemma}
\proof{}
It follows the proof of Lemma 1 in Zhang (2010). \qedsymbol
\endproof

We now determine whether a set $G(s^L, a^L)$ is pair-wise dominated by a set $G(s^L, a^{L'})$ based on the geometric relationship. Without loss of generality, assume $s^L$ is given and there is $x^0$ such that $\bar{v}^{a^L}_t(s^L, x^0) \leq \bar{v}^{a^{L'}}_t(s^L, x^0)$. Pick any $\gamma \in G(s^L, a^{L'})$ and define the set $\Phi(\gamma) \equiv \{(\lambda_1,...,\lambda_n):\gamma \geq \sum_{i=1}^n \lambda_iw^i, w^i \in G(s^L, a^L), \sum_{i=1}^n \lambda_i=1, \lambda_i \geq 0\}$. Proposition \ref{PairwiseDominanceProp} shows that determining the dominance relationship is equivalent to check whether $\Phi(\gamma)$ is empty for every $\gamma \in G(s^L, a^{L'})$. The detailed pseudocode is summarized in Algorithm \ref{pairwise}.

\begin{proposition}
	$G(s^L, a^L)$ is dominated by $G(s^L, a^{L'})$ if and only if for each $\gamma \in G(s^L, a^{L'})$, the set $\Phi(\gamma)$ is non-empty.
	\label{PairwiseDominanceProp}
\end{proposition}
\proof{}
Assume for each $\gamma \in G(s^L, a^{L'})$, $\Phi(\gamma)$ is non-empty. Equivalently, $\forall \gamma \in G(s^L, a^{L'})$, there is $\omega_\gamma \in Co(G(s^L, a^{L}))$ such that $\gamma \geq \omega_\gamma$. Pick $\forall \hat{x}^L \in X^L$ and let $\hat{\gamma} \in \arg\min \{\hat{x}^L\gamma: \gamma \in G(s^L, a^{L'})\}$. Thus, $\bar{v}^{a^{L'}}_t(s^L, \hat{x}^L) = \hat{x}^L\hat{\gamma} \geq \hat{x}^L\omega_{\hat{\gamma}}$. Lemma \ref{geometricLemma} guarantees that there is a $\hat{\omega} \in NCo(G(s^L, a^{L}))$ satisfying $\bar{v}^{a^{L}}_t(s^L, \hat{x}^L)=\hat{x}^L\hat{\omega} \leq \hat{x}^L\omega_{\hat{\gamma}}$. Thus, $\bar{v}^{a^{L}}_t(s^L, x^L) \leq \bar{v}^{a^{L'}}_t(s^L, x^L), \forall x^L \in X^L$ and the set $G(s^L, a^L)$ is dominated by the set $G(s^L, a^{L'})$. 

Conversely, let $\gamma^* \in G(s^L, a^{L'})$ be the $\gamma$-vector such that $\Phi(\gamma^*) = \emptyset$. Equivalently, $\gamma^* < \omega, \forall \omega \in NCo(G(s^L, a^{L}))$. As every $\gamma$-vector in $G(s^L,a^{L'})$ is a defining vector for $\bar{v}^{a^{L'}}$, there is $x^* \in X^L$ such that $\bar{v}^{a^{L'}}(s^L,x^*) = x^*\gamma^*$. Lemma \ref{geometricLemma} further guarantees that there is a $\omega^* \in NCo(G(s^L, a^{L}))$ satisfying $\bar{v}^{a^{L}}(s^L,x^*) = x^*\omega^*$. Thus, $\bar{v}^{a^{L'}}(s^L,x^*) < \bar{v}^{a^{L}}(s^L,x^*)$. Since $\bar{v}^{a^L}_t(s^L, x^0) \leq \bar{v}^{a^{L'}}_t(s^L, x^0)$ by assumption, both functions are continuous, and $X^L$ is connected, the two functions intersect over $X^L$. \qedsymbol
\endproof

\begin{algorithm}[h]
	\caption{Determining the Pairwise Dominance between $G(s^L,a^L)$ and $G(s^L,a^{L'})$.}	
	\begin{algorithmic}\linespread{0.83}\selectfont	
		\State{\texttt{Initialization:}}
		\Indent
		\State\parbox[t]{\dimexpr0.95\linewidth-\algorithmicindent}{\texttt{Randomly generate $x^0 \in X^L$ and calculate $\bar{v}^{a^L}(s^L,x^0)$ and $\bar{v}^{a^{L'}}(s^L,x^0)$, where $\bar{v}^{a^L}(s^L,x^L)=\min\{x^L\gamma: \gamma \in G(s^L,a^L) \}$. If $\bar{v}^{a^L}(s^L,x^0) > \bar{v}^{a^{L'}}(s^L,x^0)$, switch $G(s^L,a^L)$ and $G(s^L,a^{L'})$. }}
		\State\parbox[t]{\dimexpr0.95\linewidth-\algorithmicindent}{\texttt{Set $n=|G(s^L,a^L)|$ and PairwiseDominance=TRUE.}}
		\EndIndent	
		\For{\texttt{each $\gamma \in G(s^L,a^{L'})$}}
		
		\State \Indent\parbox[t]{\dimexpr\linewidth-\algorithmicindent}{\texttt{Check the if the set $\Phi(\lambda) = \emptyset$ where 
				$\Phi(\lambda)=\bigg\{(\lambda_1,...,\lambda_n):\gamma \geq \sum_{i=1}^n \lambda_iw^i, w^i \in G(s^L,a^{L}), \sum_{i=1}^n \lambda_i=1, \lambda_i \geq 0 \bigg\}$
		}}\EndIndent	
		\If{\texttt{$\Phi(\lambda)=\emptyset$}}
		\State{\texttt{PairwiseDominance=FALSE; break;}}
		\EndIf
		\EndFor
	\end{algorithmic}
	\label{pairwise}
\end{algorithm}

We remark that we also can determine the dominance on primal space $X^L$ by checking if $NCo(G(s^L, a^{L}))\cap NCo(G(s^L, a^{L'})) \neq \emptyset$ in the dual space, given $s^L \in S^L$. See Proposition \ref{NotIFFProp}. Conversely, however, if $\bar{v}^{a^L}_t(s^L, x^L)$ and $\bar{v}^{a^{L'}}_t(s^L, x^L)$ intersect on $X^L$, $NCo(G(s^L, a^{L}))\cap NCo(G(s^L, a^{L'}))$ could be empty in the dual space. A counterexample is given Fig. \ref{NotIFF}. 
\begin{proposition}
	$\forall s^L \in S^L$, if $NCo(G(s^L, a^{L})) \cap NCo(G(s^L, a^{L'})) \neq \emptyset$, then $\bar{v}^{a^L}_t(s^L, x^L)$ and $\bar{v}^{a^{L'}}_t(s^L, x^L)$ intersect on $X^L$.
	\label{NotIFFProp} 
\end{proposition}

\proof{}
Pick $\forall \gamma^* \in NCo(G(s^L, a^{L}))\cap NCo(G(s^L, a^{L'}))$. Let $x^* \in X^L$ such that $\bar{v}^{a^{L}}_t(s^L, x^*) = x^*\gamma^*$ per Lemma \ref{geometricLemma}. The definition of $NCo$ and Lemma \ref{geometricLemma} guarantee $\bar{v}^{a^{L'}}_t(s^L, x^*) \leq  x^*\gamma^*$. The result follows by the assumption that $\bar{v}^{a^{L'}}_t(s^L, x^0) \geq  \bar{v}^{a^{L}}_t(s^L, x^0).$ \qedsymbol
\endproof
\begin{figure}[h]
	\centering
	\begin{minipage}{0.46\textwidth}
		\centering
		\includegraphics[width=\textwidth]{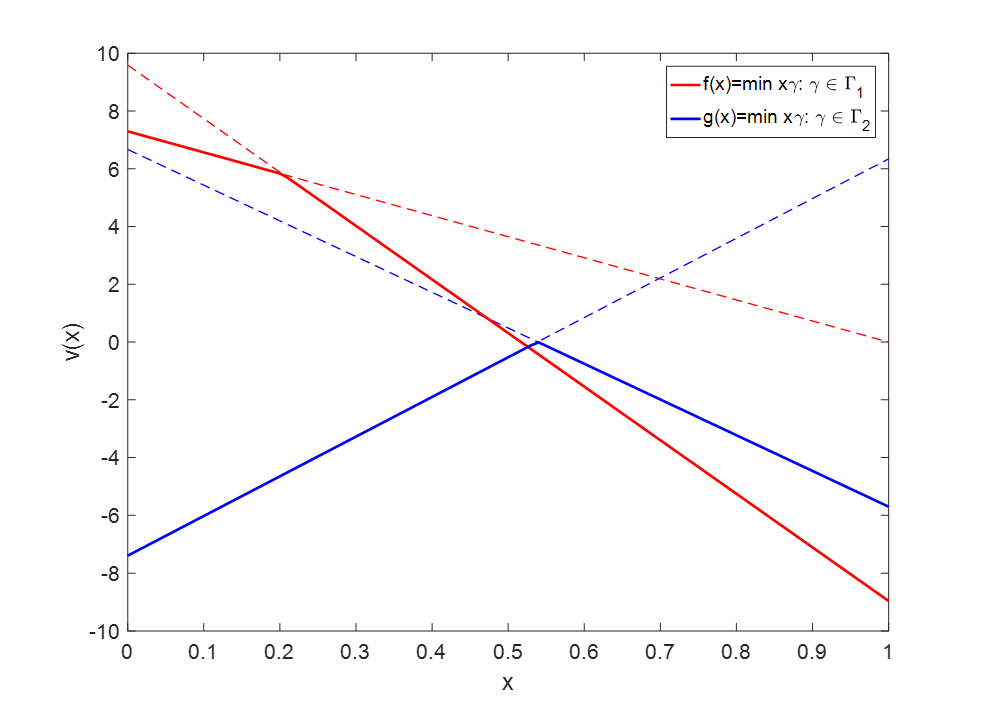}\\
		(a) primal space
	\end{minipage}
	\begin{minipage}{0.46\textwidth}
		\centering
		\includegraphics[width=\textwidth]{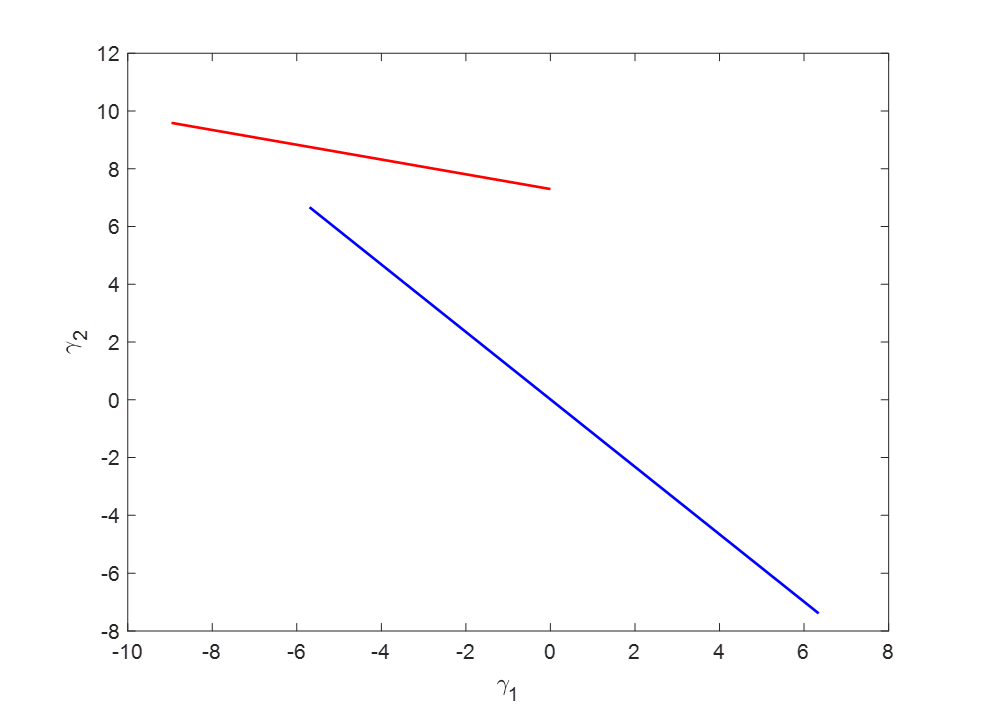}\\
		(b) dual space
	\end{minipage}\hfil
	\caption{The intersection in the (a) primal space does not imply the intersection in the (b) dual space}
	\label{NotIFF}
\end{figure}

Determining the superset $\Gamma^c(s^L)$ on the basis of pairwise dominance needs to consider each pair of action $a^L, a^{L'} \in A^L$ (the pseudocode is summarized in Algorithm \ref{geometricalgorithm2}). For each $s^L \in S^L$, the program initializes $\Gamma^c(s^L)$ with the set $G(s^L,a^{L,*})$, where $a^{L,*} \in \arg\max_{a^L \in A^L}\bar{v}^{a^L}_t(s^L,x^0)$ and $x^0 \in X^L$ is randomly generated. Let $\Gamma^c(s^L,k_1) = \{\gamma^{k_1',k_2'}: \gamma^{k_1',k_2'} \in \Gamma^c(s^L), k_1' = k_1 \}$, $K_1(s^L)$ be the number of $\Gamma^c(s^L,k_1)$ sets in $\Gamma^c(s^L)$, and $K_2(s^L,k_1)$ be the number of $\gamma$-vectors in $\Gamma^c(s^L,k_1)$. The algorithm updates $\Gamma^c(s^L), K_1(s^L)$, and $K_2(s^L,k_1)$ by the following: For each candidate set $G(s^L,a^L)$, the algorithm compares it with the existing sets in $\Gamma^c(s^L)$. If $G(s^L,a^L)$ is pair-wise dominated by an existing set $\Gamma^c(s^L,k_1)$, then $G(s^L,a^L)$ will not be considered; Otherwise, $G(s^L,a^L)$ will be included in $\Gamma^c(s^L)$ and any existing sets $\Gamma^c(s^L,k_1)$ that are dominated by $G(s^L,a^L)$ will be eliminated from $\Gamma^c(s^L)$. Meanwhile, $K_1$ and $K_2$ are updated accordingly. 
\begin{algorithm}[h] 
	\caption{Determining the Superset $\Gamma^c(s^L) \supseteq \Gamma(s^L)$.}	
	\begin{algorithmic}\linespread{0.83}\selectfont		
		\For{\texttt{each $s^L \in S^L$}}
		\State\parbox[t]{\dimexpr\linewidth-\algorithmicindent}{{\texttt{Randomly generate $x^0 \in X^L$ and calculate $\bar{v}^{a^L}(s^L,x^0) = \min\{x^0\gamma: \gamma \in G(s^L,a^L) \}$ for each $a^L$. Select $a^{L,*} \in \arg\max_{a^L}\bar{v}^{a^L}(s^L,x^0) $. Let $\Gamma^c(s^L,1) = G(s^L,a^{L,*})$, $K_1(s^L)=1$, $K_2(s^L,1)=|G(s^L,a^{L,*})|$.\\}}}
		\For{\texttt{each $a^L \in A^L$, $a^L \neq a^{L,*}$}}
		\State\parbox[t]{\dimexpr\linewidth-\algorithmicindent}{\texttt{Randomly generate $x' \in X^L$, and calculate:
				\begin{align*}
				\bar{v}^{k_1}(s^L,x') &= \min\{x'\gamma: \gamma \in \Gamma^c(s^L,k_1) \}, \forall k_1 = 1,...,K_1;\\
				\bar{v}^{a^L}(s^L,x') &= \min\{x'\gamma: \gamma \in G(s^L,a^{L}) \}.
				\end{align*}}}
		\State{\texttt{Set IsDominance = FALSE.}}
		\For{ \texttt{each $k_1$ s.t. $\bar{v}^{k_1}(s^L,x') \geq \bar{v}^{a^L}(s^L,x')$}}
		\If{\texttt{$G(s^L,a^{L})$ is pair-wise dominated by $\Gamma^c(s^L,k_1)$}}
		\State{\texttt{Set IsDominance = TRUE; break.}}
		\EndIf
		\EndFor
		\If{\texttt{IsDominance = FALSE}}
		\State\parbox[t]{\dimexpr\linewidth-\algorithmicindent}{\texttt{$\Gamma^c(s^L)=\Gamma^c(s^L) \cup G(s^L,a^L)$, $K_1(s^L)=K_1+1$, $K_2(s^L, K_1)=|G(s^L,a^L)|$.}}
		\For{ \texttt{each $k_1$ s.t. $\bar{v}^{k_1}(s^L,x') <\bar{v}^{a^L}(s^L,x')$}}
		\If{\texttt{$\Gamma^c(s^L,k_1)$ is dominated by $G(s^L,a^L)$}}
		\State\parbox[t]{\dimexpr\linewidth-\algorithmicindent}{\texttt{$\Gamma^c(s^L) = \Gamma^c(s^L) \setminus \Gamma^c(s^L,k_1)$, $K_1=K_1-1$,delete $K_2(s^L,k_1)$.}}
		\EndIf
		\EndFor
		\EndIf
		\EndFor
		\EndFor 
		\State{\texttt{Return $\Gamma^c, K1, K2$.}}
	\end{algorithmic}
	\label{geometricalgorithm2}
\end{algorithm}

\subsection{Determine the Set $\Gamma_{t}(s^L)$}
We now determine the set $\Gamma_{t}(s^L)$ by further removing the dominated sets from $\Gamma_{t}^c(s^L)$. Assume $s^L$ is given. $\forall x^L \in X^L$, let $z_1(x^L)$ be the function value attained by the superset $\Gamma_{t}^c(s^L)$, $z_1(x^L) = \max_{k_1}\min_{k_2}\{x^L\gamma^{k_1,k_2}:\gamma^{k_1,k_2} \in \Gamma_{t}^c(s^L) \}$, and $z_2(x^L)$ be the value of function $\bar{v}^{k_1}$ attained by the set $\Gamma^c(s^L,k_1)$, $z_2(x^L) = \min\{x^L\gamma: \gamma \in \Gamma^c(s^L,k_1) \}$. Let DOMINANCE\_MIP($\Gamma^c(s^L,k_1), \Gamma_{t}^c(s^L)$) determine whether the set $\Gamma^c(s^L,k_1)$ is a dominated set, which can be evaluated via the following mixed integer program \eqref{MIP}: 
\begin{align}
u \equiv \min \hspace{15pt}         & z_1 - z_2    \nonumber \\
\mathrm{s.t.}\hspace{15pt}  &  z_2 \leq x\gamma, \gamma \in \Gamma_{t}^c(s^L,k_1);\nonumber \\
&     -z_1 \leq -x\gamma^{k_1,k_2}+M(1-\rho^{k_1,k_2}), \gamma^{k_1,k_2} \in \Gamma_{t}^c(s^L); \nonumber\\
& \sum_{k_2=1}^{K_2(s^L,k_1)}\rho^{k_1,k_2}=1, \forall k_1 = 1,...,K_1(s^L);\label{MIP}\\
& \rho^{k_1,k_2} \in \{0,1\}, x \in X^L, z_1, z_2 \in R, \nonumber
\end{align}
where $M$ is a large positive number. 

The objective function is to find the minimal gap between the two functions, $z_1(x^L)$ and $z_2(x^L)$. As $z_2(x^L)$ is a piecewise linear and concave function on $X^L$, it can be easily determined by the first constraint. The second and the third constraints define $z_1$. For the purpose of explanation, $\forall k_1 \in K_1(s^L)$, we further introduce a variable $\eta^{k_1}$ as the minimum value attained by set $\Gamma_{t}^c(s^L,k_1)$, i.e., $\eta^{k_1}(x^L) = \min\{x^L\gamma: \gamma \in \Gamma_{t}^c(s^L,k_1)\}$. Then, (i) $\eta^{k_1} \geq x^L\gamma^{k_1,k_2} - M(1-\rho^{k_1,k_2})$ and the multiple-choice constraint on $\rho^{k_1,k_2}$s ensure that there is exactly one $\gamma \in \Gamma_{t}^c(s^L,k_1)$ selected to define $\eta^{k_1}$; (ii) $z_1 \geq \eta^{k_1}, \forall k_1$, by the definition of $z_1$. Hence, the combination of (i) and (ii) leads to the second constraint and variables $\eta^{k_1}$s can be omitted. The last equation ensures that the belief states are in a nonnegative simplex.  

Clearly, if the objective value $u > 0$, then $\Gamma_{t}^c(s^L,k_1)$ is not a supporting set for $\bar{v}_t^L$ and should be eliminated. We need to solve $K_1$ number of MIPs to finalize $\Gamma_{t}(s^L)$. The pseudocode for determining $\Gamma_{t}(s^L)$ from its superset $\Gamma_{t}^c(s^L)$ is presented in Algorithm \ref{MIPDominancealgorithm}. 

\begin{algorithm}[h]
	\caption{Determining $\Gamma_{t}(s^L)$ from its Superset $\Gamma_{t}^c(s^L)$.}
	\begin{algorithmic}\linespread{0.83}\selectfont		
		\For{\texttt{each $s^L \in S^L$}}
		\State{\texttt{Set $\Gamma_{t}(s^L)=\Gamma_{t}^c(s^L)$.}}
		\For{\texttt{each $ k_1 \in K_1(s^L)$}}
		\State \parbox[t]{\dimexpr\linewidth-\algorithmicindent}{\texttt{determine 
				$u=$DOMINANCE\_MIP($\Gamma_{t}^c(s^L,k_1), \Gamma_{t}(s^L))$).} }
		\If{\texttt{$u > 0$}} 
		\State \texttt{$\Gamma_{t}(s^L) = \Gamma_{t}(s^L) \setminus \Gamma_{t}^c(s^L,k_1)$.}
		\EndIf
		\EndFor
		\EndFor
		\State \texttt{Return $\Gamma_t$.} 
	\end{algorithmic}
	\label{MIPDominancealgorithm}
\end{algorithm} 

\section{Piecewise Linear Concave Approximation}
\label{PiecewiseLinearConcaveApproximation}
Given a set of $\gamma$-vectors $\Gamma_t$, the value function $\bar{v}_t^L(s^L,x^L) = \max_{k_1}\min_{k_2} \{x^L\gamma^{k_1, k_2}: \gamma^{k_1, k_2} \in \Gamma_t(s^L) \}$ is piecewise linear but not concave. The iterative algorithm we developed requires a piecewise linear and concave function for the next iteration. We thus approximate $\bar{v}_t^L$ by a function $\tilde{v}_t^L$ satisfying the following conditions:
\begin{enumerate}[(i)]
	\item $\forall s^L \in S^L, \tilde{v}_t^L(s^L, x^L)$ is a piecewise linear and concave on $X^L$;
	\item  $\tilde{v}_t^L(s^L, x^L) \leq\bar{v}_t^L(s^L, x^L), \forall s^L \in S^L, x^L \in X^L$;
	\item  the distance between $\tilde{v}_t^L$ and $\bar{v}_t^L$ is as small as possible, where we define the distance between two bounded functions $v^1, v^2 \in V$ as $$d(v^1, v^2)(s^L) = \max_{x^L\in X^L}|v^1(s^L,x^L)-v^2(s^L,x^L)|.$$
\end{enumerate}
Equivalently, for each $s^L$, we want to determine a set $\tilde{\Gamma}_t(s^L)$ so that $\tilde{v}_t^L(s^L, x^L) = \min\{ x^L\gamma: \gamma \in \tilde{\Gamma}(s^L) \}$ satisfies conditions (ii) and (iii). For computational efficiency, we consider the case where $\tilde{\Gamma}_t(s^L) \subset \Gamma_t(s^L)$ in this paper. We do acknowledge that $\tilde{v}_t^L(s^L, x^L)$ may be further improved by constructing $\gamma \notin \Gamma_t(s^L)$ for some instances. Determining a general procedure for finding the best piecewise linear and concave approximation of an arbitrary piecewise linear function is an interesting research topic for the future. Furthermore, an advantage of selecting $\tilde{\Gamma}_t(s^L) \subset \Gamma_t(s^L)$ is that each $\gamma \in \tilde{\Gamma}_t(s^L)$ is still associated with an action pair $a=(a^L,a^F)$. Thus, it is easy to explain and implement the policy associated with the lower bound $\tilde{v}_t^L(s^L, x^L)$. 

We remark that the maximal gap between two functions $v^1, v^2 \in V$ for a given $s^L \in S^L$ must occur at (i) where two segments of $v^1$ (or $v^2$) intersect, or (ii) extreme points of $X^L$. Thus, $\forall s^L \in S^L$, we could determine the set $\tilde{\Gamma}_t(s^L)$ satisfying conditions (ii) and (iii) by a finite set of belief points $W \subset X^L$. Given $s^L \in S^L$, the pseudocode of determining $\tilde{\Gamma}_t(s^L)$ is outlined in Algorithm \ref{PLCA}. 

\begin{algorithm}[h]
	\caption{Approximating $\bar{v}_t^L(s^L, x^L)$ by $\tilde{v}_t^L(s^L, x^L)$}
	\centering
	\begin{algorithmic}\linespread{0.83}\selectfont		
		\State \texttt{Step 1: Initialize $W_0$ by including the following two groups:
			\begin{enumerate}
				\item[] \textit{Extreme points}: extreme points of $X^L$ are $e_i, i \in S^F$, whose $i^{th}$ entry is 1; 0 elsewhere. Evaluate $\bar{v}_t^L(s^L,e_i) = \max_{k_1}\min_{k_2}\{e_i\gamma^{k_1,k_2}: \gamma^{k_1,k_2} \in \Gamma_t(s^L) \}$.
				\item[] \textit{Witness points:} the PURGE operation has identified at least a witness point $w^i$ for each $\gamma^i \in \Gamma_t(s^L)$. Evaluate $\bar{v}_t^L(s^L,w^i) = \max_{k_1}\min_{k_2}\{w^i\gamma^{k_1,k_2}: \gamma^{k_1,k_2} \in \Gamma_t(s^L) \}$.
				\item[] Let $N = |W_0|$. 
		\end{enumerate}}
		
		\State \parbox[t]{\dimexpr\linewidth-\algorithmicindent}{	\hangindent=10pt{\texttt{Step 2: Construct the concave approximation set $\tilde{\Gamma}_n(s^L)$ by the concave approximation MIP \eqref{ApproximationMIP} on the set $W_n$.\\} }}
		\State \parbox[t]{\dimexpr\linewidth-\algorithmicindent}{	\hangindent=10pt{\texttt{Step 3: Check if the condition (ii) is satisfied on $X^L$ by the verification MIP \eqref{ApproximationMIPCheck}. If the condition (ii) is violated, the verification MIP will return an $x^* \in X^L$ with $\tilde{v}_t^L(s^L, x^*) > \bar{v}_t^L(s^L, x^*)$. Add $x^*$ to the set $W_n$.\\} }}
		\State \parbox[t]{\dimexpr\linewidth-\algorithmicindent}{	\hangindent=10pt{\texttt{Step 4: Evaluate the maximal gap between $\tilde{v}_t^L(s^L, x^L)$ and $\bar{v}_t^L(s^L, x^L)$ by the MIP \eqref{BoundErrorMIP}. If the maximal gap $\epsilon(s^L)$ is positive at point $x' \in X^L$, $W_{n+1} = W_n \cup \{x'\}$ and update $N=|W_{n+1}|$.\\}}}
		\State \parbox[t]{\dimexpr\linewidth-\algorithmicindent}{	\hangindent=10pt{\texttt{Step 5: Go to Step 2 and update the concave approximation set $\tilde{\Gamma}_{n+1}(s^L)$ on the set $W_{n+1}$. The program stops if $W_{n+1} = W_n$. The difference between $\tilde{v}_t^L(s^L, x^L)$ and $\bar{v}_t^L(s^L, x^L)$ is bounded by $\epsilon(s^L)$. \\}}}
	\end{algorithmic}
	\label{PLCA}
\end{algorithm}

We initialize the set $W_0$ in Step 1, by including the extreme points of the belief space $X^L$ and at least a witness point for each $\gamma$-vector in $\Gamma_t(s^L)$. These witness points are generated by the PURGE operation discussed in Section \ref{PurgeOperationSection}. We develop a concave approximation MIP in Step 2 to construct an initial set $\tilde{\Gamma}_t(s^L)$ based on $W_n$. As the condition (ii) is only enforced on the set $W_n$ in Step 2, Step 3 further determines if the condition (ii) is violated on $X^L$. If there is an $x^* \in X^L$ at which $\tilde{v}_t^L(s^L, x^*) > \bar{v}_t^L(s^L,x^*)$, we update $W_n$ by including $x^*$. Step 4 determines $\epsilon(s^L)$, the maximal distance between the $\tilde{v}_t^L(s^L, x)$ and $\bar{v}_t^L(s^L,x)$. To improve the approximation quality and reduce the gap between $\bar{v}_t^L$ and $\tilde{v}_t^L$, we also add the belief point at which the maximal distance is attained to the new set $W_{n+1}$. The program continues to update $\tilde{\Gamma}_t(s^L)$ based on $W_{n+1}$. The entire procedure stops when no further improvement is identified. When it stops, the condition (ii) is guaranteed on $X^L$ and the maximal distance between $\bar{v}(s^L)$ and its approximate value $\tilde{v}(s^L)$ is bounded above by $\epsilon(s^L)$. We now detail each step in the following subsections.

\subsection{Concave Approximation on $W$}
\label{SecConcaveApproximationMIP}
Assume $s^L \in S^L$ is given. Let $W=\{x^i\}$ be a set of belief points in $X^L$, $N=|W|$, $\tilde{z}^i$ be the maximum function values attained at $x^i$ by the set $\tilde{\Gamma}_t(s^L)$, and $z^i$ be the function values attained at $x^i$ by the set $\Gamma_t(s^L)$, i.e., $\tilde{z}^i = \tilde{v}_t^L(s^L,x^i) = \min\{x^i\gamma: \gamma \in \tilde{\Gamma}_t(s^L)\}$, and $z^i=\bar{v}_t^L(s^L,x^i) = \max_{k_1}\min_{k_2}\{x^i\gamma^{k_1,k_2}: \gamma^{k_1,k_2} \in \Gamma_t(s^L)\}$. 

For each $\gamma^{k_1,k_2} \in \Gamma_t(s^L)$, define a binary variable $y^{k_1,k_2} = 1$ if $\gamma^{k_1,k_2} \in \tilde{\Gamma}_t(s^L)$ and $y^{k_1,k_2} = 0$ if $\gamma^{k_1,k_2} \notin \tilde{\Gamma}_t(s^L)$. Thus, $\tilde{\Gamma}_t(s^L) = \{\gamma^{k_1, k_2}: \gamma^{k_1,k_2} \in \Gamma_t(s^L), y^{k_1,k_2}=1 \}$. Let $g$ be the distance between $\bar{v}_t^L(s^L,x^L)$ and $\tilde{v}_t^L(s^L,x^L)$ on the set $W$, that is, $g = \max_i z^i - \tilde{z}^i$. With the aid of additional binary variables for evaluating $\tilde{z}^i$, we seek the set $\tilde{\Gamma}_t(s^L)$ by the following mixed integer program \eqref{ApproximationMIP}:
\begin{align}
min \hspace{15pt}         & Ng-\sum_{i} \tilde{z}^i    \nonumber \\
\mathrm{s.t.}\hspace{15pt}  &  \tilde{z}^i \leq x^i\gamma^{k_1,k_2} + M(1-y^{k_1,k_2}), \gamma^{k_1,k_2} \in \Gamma_t(s^L),  \forall i;\nonumber \\
& -\tilde{z}^i \leq -x^i \gamma^{k_1, k_2} + M(1- \eta_i^{k_1, k_2}), \forall i, k_1, k_2; \nonumber\\
& \sum_{k_1}\sum_{k_2} \eta_i^{k_1,k_2} = 1, \forall i; \nonumber\\
& \eta_i^{k_1,k_2} \leq y^{k_1, k_2}, \forall i, k_1, k_2;\label{ApproximationMIP}\\
&  1 \leq \sum_{k_1=1}^{|K_1|}\sum_{k_2=1}^{|K_2|}y^{k_1,k_2} \leq |\Gamma|-1; \nonumber \\
& \tilde{z}^i \leq z^i, \forall i; \nonumber\\
& g \geq z^i - \tilde{z}^i, \forall i; \nonumber\\
& y^{k_1,k_2}, \eta_i^{k_1, k_2} \in \{0,1\}, \tilde{z}^i,g \in R, \nonumber
\end{align}
where $M$ is a large positive number. 

Minimizing the distance between $\bar{v}_t^L$ and $\tilde{v}_t^L$ (on $W$) is equivalent to minimize the maximal gap $g$. The expression $-\sum_{i \in I} \tilde{z}^i$ is added to the objective function in order to close the gap on $W$. The multiplier $N$ on $g$ is to ensure that the two quantities are within the same magnitude. The first to the fourth constraints compute $\tilde{z}^i, \forall i$. Specifically, the first constraint ensures that $\tilde{z}^i$ is bounded above by the approximation function constructed by $\tilde{\Gamma}_t(s^L)$. Each binary variable $\eta_i^{k_1,k_2}$ is associated with a $\gamma$-vector in $\Gamma_t(s^L)$ and a belief point $x^i$. The second and the third constraints are necessary to guarantee that $\forall x^i \in X^L$, there exists one and only one defining vector $\gamma^{k_1,k_2} \in \tilde{\Gamma}_t(s^L)$ such that $\tilde{z}^i = x^i \gamma^{k_1,k_2}$. The fourth constraint implies that if $\gamma^{k_1,k_2}$ satisfies $\tilde{z}^i = x^i \gamma^{k_1,k_2}$, then $\gamma^{k_1,k_2} \in \tilde{\Gamma}_t(s^L)$. The fifth constraint is based on the observation that $\min_{k_1}\min_{k_2}\{x^L\gamma^{k_1,k_2}: \gamma^{k_1,k_2} \in \Gamma_t(s^L)\} \leq \min_{k_2}\{x^L\gamma^{k_1,k_2}: \gamma^{k_1,k_2} \in \Gamma_t(s^L)\}, \forall k_1$, hence, $\tilde{\Gamma}_t(s^L) \subsetneq \Gamma_t(s^L)$. The second to the last constraint guarantees that $\tilde{v}_t^L(s^L,x^L) \leq \bar{v}_t^L(s^L,x^L)$ on $W$ and the last constraint determines the maximal gap between $\tilde{v}_t^L(s^L,x^L)$ and $\bar{v}_t^L(s^L,x^L)$ on $W$.  

We can enhance the performance of the MIP \eqref{ApproximationMIP} by providing a good feasible solution exploiting the structure results of $\bar{v}_t^L$. Note that for any given $k_1$, $\bar{v}_t^{k_1}$ computed by the set $\Gamma_t(s^L,k_1)$ is a lower bound to $\bar{v}_t^L(s^L,x^L)$ and satisfies all three conditions. Pick any $w \in W$. Let $z_w^{k_1} = \min \{w\gamma: \gamma \in \Gamma_t(s^L,k_1)\}$ and $k'_1 \in \arg\max_{k_1}z^{k_1}_w$. Then $\Gamma_t(s^L,k'_1)$ is a feasible solution to the MIP \eqref{ApproximationMIP}. Determining such initial solutions is straightforward and computationally inexpensive.  
\subsection{Verification on $X^L$}
The concave approximation MIP \eqref{ApproximationMIP} only ensures that the  condition (ii) is satisfied on $W\subsetneq X^L$. The following mixed integer program \eqref{ApproximationMIPCheck} further checks whether the condition is satisfied on $X^L$: 
\begin{align}
\mu^* \equiv min \hspace{15pt}         & z_1-z_2   \nonumber \\
\mathrm{s.t.}\hspace{15pt}  &  z_2 \leq x\gamma, \gamma \in \tilde{\Gamma}_t(s^L);\nonumber \\
& -z_1 \leq -x \gamma^{k_1, k_2} + M(1- \rho^{k_1, k_2}), \gamma^{k_1,k_2} \in \Gamma_t(s^L); \nonumber\\
& \sum_{k_2=1}^{K_2(s^L,k_1)}\rho^{k_1,k_2} = 1, \forall k_1; \label{ApproximationMIPCheck}\\
& \rho^{k_1,k_2} \in \{0,1\}, z_1,z_2 \in R, x \in X^L, \nonumber
\end{align}
where $M$ is a large positive number. 

The objective function is to minimize the difference between the two functions for a given $s^L \in S^L$: $z^1(x^L)=\max_{k_1}\min_{k_2}\{x^L\gamma^{k_1,k_2}: \gamma^{k_1,k_2} \in \Gamma_t(s^L)\}$ and its approximation $z^2(x^L)=\min\{x^L\gamma: \gamma \in \tilde{\Gamma}_t(s^L)\}$. Thus, MIP \eqref{ApproximationMIPCheck} is the same as MIP \eqref{MIP} where: (i) the value $z_2$ is determined by the first constraint, and (ii) the second and the third constraints and the binary variable $\rho^{k_1,k_2}$ associated with each vector $\gamma^{k_1,k_2} \in \Gamma_t(s^L)$ determine $z_1$.

If $\mu^* < 0$ at the belief state $x^* \in X^L$, then $x^*$ should be added to $W$, and both of the MIPs \eqref{ApproximationMIP} and \eqref{ApproximationMIPCheck} should be resolved. The process should continue until $\mu^* \geq 0$.     

\subsection{Approximation Error}
\label{ApproximationError}
We now determine $\epsilon(s^L)$, the maximal difference between $z_1(x^L)$ based on $\Gamma_{t}(s^L)$ and its approximation $z_2(x^L)$ based on $\tilde{\Gamma}_{t}(s^L)$, by the following MIP \eqref{BoundErrorMIP}:
\begin{align}
\epsilon(s^L)  \equiv max \hspace{15pt}         & z_1-z_2   \nonumber \\
\mathrm{s.t.}\hspace{15pt}  &  z_1 \leq x\gamma^{k_1,k_2}+M(1-y^{k_1}), \forall k_1, k_2;\nonumber \\
& \sum_{k_1=1}^{K_1}y^{k_1} =1; \nonumber\\
& -z_2 \leq -x \gamma^{k} + M(1- \rho^{k}), \gamma^k \in \tilde{\Gamma}_{t}(s^L);  \label{BoundErrorMIP}\\
& \sum_{k}\rho^{k} = 1;\nonumber\\
& y^{k_1}, \rho^k \in \{0,1\}, z_1,z_2 \in R, x \in X^L, \nonumber
\end{align}
where $M$ is a large positive number. 

The objective function is to find the maximal gap between $z_1(x^L)$ and $z_2(x^L)$. The first two constraints compute $z_1(x^L)$ on the basis of $\Gamma_t(s^L)$. As $z_1(x^L) = \max_{k_1}v^{k_1}(s^L,x^L)$ for a given $s^L$, the binary variable $y^{k_1}$ for each $k_1$ and the multiple-choice constraint on $y^{k_1}$ ensure that there is exactly one $k_1$ selected to compute $z_1$. Meanwhile, $z_1(x^L) \leq x^L\gamma^{k_1,k_2}, \forall k_2$ for the selected $k_1$. Similarly, the third and the fourth constraints compute $z_2(x^L)$. The binary variable $\rho^k$ associated with each $\gamma$-vector in $\tilde{\Gamma}(s^L)$ and its multiple-choice constraint guarantee that there exists one and only one $\gamma \in \tilde{\Gamma}(s^L)$ defining $z_2(x^L)$. 

The approximation error is bounded above by the objective value $\epsilon(s^L) \geq 0$, assuming at point $x^* \in X^L$. To improve the approximation quality, we also include $x^*$ to update $W$ and $\tilde{\Gamma}(s^L)$. Let $2^{nd}best$ be the second best leader's action for $\min_{k_1} \sum_{x^i \in W} |\bar{v}(s^L,x^i)-\bar{v}^{k_1}(s^L,x^i)| $. The following Proposition shows that this procedure guarantees $\epsilon(s^L) \leq \max_{x \in X^L}\max_{k_1} |\bar{v}^{k_1}(s^L,x^L)-\bar{v}^{2^{nd}best}(s^L,x^L)|$. That is, the approximation function at any belief point is no worse than the performance induced by the leader's second best action (on W). Moreover, the approximation error of the proposed approach could be zero when there is a dominant action of the leader. 

\begin{proposition}
	$\epsilon(s^L) \leq \max_{x \in X^L}\max_{k_1} |\bar{v}^{k_1}(s^L,x^L)-\bar{v}^{2^{nd}best}(s^L,x^L)|$. Furthermore, if there is a leader's action $a^{L,*}$ such that $G(s^L, a^{L,*})$ pairwisely dominates $G(s^L, a^{L}), \forall a^{L} \in A^L, a^L \neq a^{L,*}$, then $\epsilon(s^L) = 0$. 
	\label{CorollaryApproximationErrosisZero}
\end{proposition}

\proof{}
The first result follows from that (i) $\bar{v}^{2^{nd}best}(s^L,x^L)$ is a feasible solution to the MIP (2); (ii) it is a second best minimizer of $\min_{\bar{v}^{k_1}}\sum_{i \in W}[\bar{v}(s^L,w_i)-\bar{v}^{k_1}(s^L,w_i)]$; and (iii) the construction of the set $W$ (Step 4 in Algorithm \ref{PLCA}).
The second result follows as the set $G(s^L, a^{L,*})$ satisfies the conditions (i)-(iii) and the pairwise dominance assumption implies $\tilde{\Gamma}_t(s^L) = \Gamma_{t}(s^L)=G(s^L,a^{L,*})$.
\qedsymbol 
\endproof

Furthermore, let $T:V\rightarrow V$ be the (nonlinear) operator  such that $\forall u \in V$, $Tu$ is the approximation of $u$ satisfying the conditions (i)-(iii), and $\tilde{v}_T^L = Tv^L_T$. At each iteration, Algorithm \ref{overall} evaluates $\bar{v}_t^L = H\tilde{v}^L_{t+1}$ and approximates $\bar{v}_t^L$ by $\tilde{v}_{t}^L=T\bar{v}^L_t$. Thus, $\bar{v}_t^L = (H\circ T)\bar{v}_{t+1}^L. $
\begin{proposition}
	$\tilde{v}^L_t \leq	\bar{v}^L_t \leq v_t^L $ and $||\bar{v}_t^L-\bar{v}_{t-1}^L|| \leq \beta^{T-t}||\bar{v}_T^L-\bar{v}_{T-1}^L||$.
	\label{convergeThm}
\end{proposition} 

\proof{}
$v_t^L \geq \bar{v}_t^L \geq \tilde{v}_t^L$ is obvious by the definition of $T$ and the fact that if $u \leq v, u,v \in V$, then $Hu \leq Hv$. For the second part, note if $||u-v||=\epsilon$, then $u-\epsilon \leq v \leq u+\epsilon$. By the definition of $T$, $Tu-\epsilon \leq Tv \leq v \leq u+\epsilon$. Similarly, we also have $Tv-\epsilon \leq Tu \leq u \leq v + \epsilon$. Thus, $||Tu-Tv|| \leq ||u-v||$. Now, $||\bar{v}_t^L - \bar{v}_{t-1}^L|| = ||(H\circ T)\bar{v}_{t+1}^L - (H \circ T) \bar{v}_t^L|| \leq \beta ||T\bar{v}_{t+1}^L-T\bar{v}_{t}^L|| \leq \beta ||\bar{v}_{t+1}^L-\bar{v}_{t}^L|| \leq ... \leq \beta^{T-t}||\bar{v}_{T}^L-\bar{v}_{T-1}^L||$.
\qedsymbol
\endproof

Thus, if we solve the finite horizon problem for a larger and larger $T$, the lower bound function $\bar{v}_t^L$ will also converge (and will be a lower bound of the fixed point $v^*=Hv^*$). Although we only focus on the finite horizon case in this paper, this result shows that the developed algorithm can also be used to obtain, at least approximately, a lower bound of leader's value function in the infinite planning horizon. We remark that there is no algorithm for the infinite-horizon general-sum POSG in the literature yet. 

\section{A Security Application}
\label{NumericalResults}
In this section, we describe a class of \textit{dynamic resource allocation} problems in security context, where the developed model and solution procedure can be employed. 

Consider a defender with limited defensive resources is protecting a set of critical nodes against an intelligent adversary over time. A node could be a manufacturing/factory site, a computer on a network, or a security checkpoint in an airport. At any time, an adversary can choose to breach the security of any node at certain levels. The objective of the defender is to minimize the number of breaches and losses generated by these violations. However, it can be difficult to quantify the reward structure of the adversary, its value over each node, and its rationality. Moreover, due to limited resources and capabilities, the defender may not be fully aware of the attacker state (e.g., exact locations, attack capabilities); but the defender may infer the attacker's state through reported locations, historical attack records, screening, sensors and detectors, and unstructured text data from social media. On the other hand, the defender's state is also only partially observable to the adversary in many realistic scenarios (e.g., some defensive resources can be camouflaged). While attacking the system, the adversary may adjust its behavior and target based on its updated information on the defensive resource allocation. In order to provide decision support to the defender, we could model such problems as a leader-follower POSG under the worst-case scenario, where the defender is the leader and the adversary is the follower. Furthermore, our solution procedure can determine dynamic defense policies that timely adjust resource allocation on the basis of all available real-time information in order to protect the system. 

Dynamic defense policies are a class of defense policies aligned with the trending philosophy of ``Moving Target Defense" (MTD) (Miehling et al. 2015). Currently, the static configuration and operation of a system has presented adversaries with an incredible advantage as adversaries can take their time to study the system and plan attacks (Department of Homeland Security 2015). The focus of MTD is to \textit{dynamically} change a system in order to shift or reduce the attack surface that can be exploited by adversaries to attack the system (Zhuang et al. 2014).  

Many concrete examples in the dynamic risk assessment and security game literature are in this application area, including cybersecurity, traveling inspection (Ahmadi et al. 2018; Bakir and Kardes 2009; Haskell et al. 2014; Kardes 2014; Lopez et al. 2013; Poolsappasit et al. 2012; Shameli-Sendi et al. 2012; Yang et al. 2014; Wang et al. 2019).  Here, we use the liquid egg production problem presented in Zhang (2013) for illustration. For the completeness of the paper, we briefly restate the security problem in Section \ref{liquidEggProblem}.

\subsection{Problem Description}
\label{liquidEggProblem}
Liquid egg products are widely used by the food service industry and as ingredients in other food products such as mayonnaise and ice cream (United States Department of Agriculture Food Safety and Inspection Service 2015). A deliberate contamination in the liquid egg products by an adversary will breach food safety, leading to excessive morbidity and mortality. Zhang (2013) identified the critical components of a liquid egg production process, including collecting vats, raw material tanks, pasteurization, and finished product tanks (Fig. \ref{liquidprocess}). An unknown adversary may use this system as a toxin delivery vehicle by inserting a toxin (e.g., botulinum) at these components (``targets"). The consequence of such attacks occurred at each component is defined as the number of contaminated packages, and the numerical values of the consequence have been analyzed in literature. 

\begin{figure}[h]
	\centering
	\includegraphics[width=0.75\textwidth]{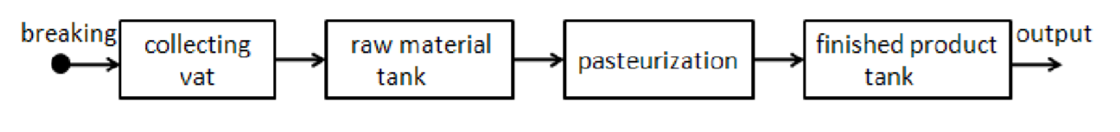}
	\caption{Critical components of the liquid egg product process identified in Zhang (2013)}
	\label{liquidprocess}
\end{figure}

We now illustrate how to use the developed method to support the production manager with limited resources in selecting a sequence of actions for protecting the system against an unknown adversary, in order to maximize the long-run productivity of the production facility. We allow for multiple attacks and each attack can be successful or unsuccessful. An unsuccessful attack occurs when the adversary launches an attack but fails to insert any toxin to the system (e.g., the adversary is caught by the manager during the attack). Thus, the production process will not be affected, and the manager needs to prepare for next possible attacks. After a successful attack, however, the manager has to stop the production process to remove inserted toxin and clean up the system. Thus, the game stops whenever a successful attack occurs. As the pasteurization process can significantly reduce the effectiveness of the botulinum toxin, we assume the manager needs to protect three targets: Collecting Vat (Target 1), Raw Production Tank (Target 2), and the Finished Product Tank (Target 3).

\textit{State spaces, action spaces, and system dynamics}: We assume the manager can only protect one target each time (e.g., visit and inspect one critical component each time). Thus, the state of the manager is the target under protection. The state space of the manager is $S^L = \{\mbox{Target 1 Protected}, \mbox{Target 2 Protected},\mbox{Target 3 Protected},\\\mbox{Attacked}\}$, where the ``Attacked" state indicates toxin has been successfully inserted to the system. The manager decides which target to protect dynamically based on its own information data. 

The state of the adversary is the location of the adversary. Hence, $S^F$= \{Target 1, Target 2, Target 3, Attacked\}, and $|S|=|S^F||S^L|=16$. At Target $i$, the adversary can either attack the target or switch to another target. Thus, there are 3 actions for each agent (9 action pairs) in each state. 

The system transits to a new state once both the agents have selected actions.

\textit{Observation space}: The manager's observations of the adversary are the possible locations of the adversary. Thus, $Z^L=S^F$, $|Z^L| = 4$. We assume that the manager has the ability to detect an attack (e.g., by testing) if the attack has successfully occurred. Specifically, the observation matrix is $P(z^L|s^F)=\epsilon_{s^F,z^L}$, where $0 \leq \epsilon_{s^F,z^L} \leq 1$ and $\sum_{z^L \in Z^L}\epsilon_{s^F,z^L} = 1$. 

\textit{Reward structure, criterion, and objective}: The system can produce $L$ number of qualified packages under normal operations. A successful attack with 2000 grams of botulinum at location $i$ can result in $L^i_{c}$ number of contaminated packages, $i \in$ \{Target 1, Target 2, Target 3\}. The numerical values of $L$ and $L^i_c$ are estimated by the simulation model developed in Zhang (2013). If there is no attack, the reward of the manger $r^L(s,a), s^L \neq \mbox{``Attacked"}$ is its normal productivity $L$. If a successful attack has been detected, no package will be produced because the manager has to stop the production and clean up the system. We assume the manager will receive an additional bonus $b>0$ for successfully preventing an attack. Let $p$ ($q$) be the probability of having a successful attack at a protected (unprotected) target. Assume $0\leq p<<q\leq 1$. The cleanup cost is $C$ for the manager to remove toxin from the system after a successful attack. Thus, 
$$r^L(s,a) = \begin{cases}
L & s^L \neq \mbox{``Attacked"}, a^F \neq \mbox{``attack"};\\
L - p*L^i_c +(1-p)b& a^L=\mbox{``protect target $i$"},  a^F = \mbox{``attack target $i$"},\\
&s^F=i, s^L \neq \mbox{``Attacked"};\\
L - q*L^i_c+(1-q)b & a^L\neq\mbox{``protect target $i$"},  a^F = \mbox{``attack target $i$"},\\
&s^F=i, s^L \neq \mbox{``Attacked"};\\
C & s^L \neq \mbox{``Attacked"},  s^F = \mbox{``Attacked Target $i$"};\\
& a^L = \mbox{``clean the system"};\\
0 & s^L = \mbox{``Attacked"}.
\end{cases}$$   

The criterion of the manager is the expected finite horizon total discounted reward $v^L(\zeta_0^L) =  E\{ \sum_{t=0}^T \beta^t r^L(s_t,a_t)|\zeta_0^L \}$, where we assume $\beta = 0.85$ and $T=30$ for illustrative purpose. The objective of the manager is to maximize the value of criterion under the worst-case scenario. 

\subsection{Numerical Results}

We first use $t=T$ to illustrate the procedure in Algorithm \ref{overall}. Table \ref{vecTable} summarizes the $\gamma$-vectors after the PURGE and DOMINANCE operations for $s^L \neq $``Attacked" at $t=T$. Thus, $v_T^L(s^L,x^L) = \max_{k_1}\min_{k_2}\{x^L\gamma^{k_1,k_2}: \gamma \in \Gamma_T(s^L)\}$. Fig. \ref{numt_0} shows the graph of the true value function $v_T^L$ and its approximation $\tilde{v}_T^L$ projected on the non-absorbing states of the follower (i.e., $s^F \neq $``Attacked"). Clearly, $v_T^L(s^L,x^L)$ (in blue) is not a concave function and $\tilde{v}_T^L(s^L,x)$ (in red) is indeed the best piecewise linear concave approximation function of $v_T^L(s^L,x)$. Let $P \subset X^L$ be the region where the approximation is accurate, i.e., $P=\{x \in X^L: v_T^L(s^L,x^L) = \tilde{v}_T^L(s^L,x^L) \}$. Then $|P|/|X^L| = 78.63\%$ (in terms of the Lebesgue measure), and the maximal approximation error ($4.36\%$) occurs around the extreme point $e_2$. 

Fig. \ref{numt} shows the convergence result of the overall procedure. The maximum deviation of the value function $\bar{v}_{t+1}^L$ from $\bar{v}_{t}^L$ $$dev = \max_{s^L \in S^L}\max_{x^L \in X^L}|\bar{v}_{t+1}^L(s^L,x^L) - \bar{v}_{t}^L(s^L,x^L) |$$ declines as the algorithm proceeds, and the value function $\bar{v}^L$ converges to $\bar{v}^{L,*}$ after 27 iterations. 
\begin{table}[ht]
	\begin{minipage}{0.46\textwidth}
		\centering
		\begin{tabular}{|cc:cc:c|}
			\hline 
			$\gamma^{1,1}$ & $\gamma^{1,2}$ & $\gamma^{2,1}$ & $\gamma^{2,2}$ & $\gamma^{3,1}$\\
			\hline
			916 & 906 & 0 & 756 & 0 \\
			723 & 906 & 916 & 756 & 703 \\
			746 & 906 & 786 & 756 & 726\\
			-100& -100 & -100 & -100 & -100\\
			\hline
		\end{tabular}	
		\vspace{20pt}
		\caption{$\Gamma_T(s^L)$ where $s^L \neq $``Attacked", in terms of the number of qualified products.}
		\label{vecTable}
	\end{minipage} \hfil
	\begin{minipage}{0.46\textwidth}
		\centering
		\includegraphics[width=0.9\textwidth]{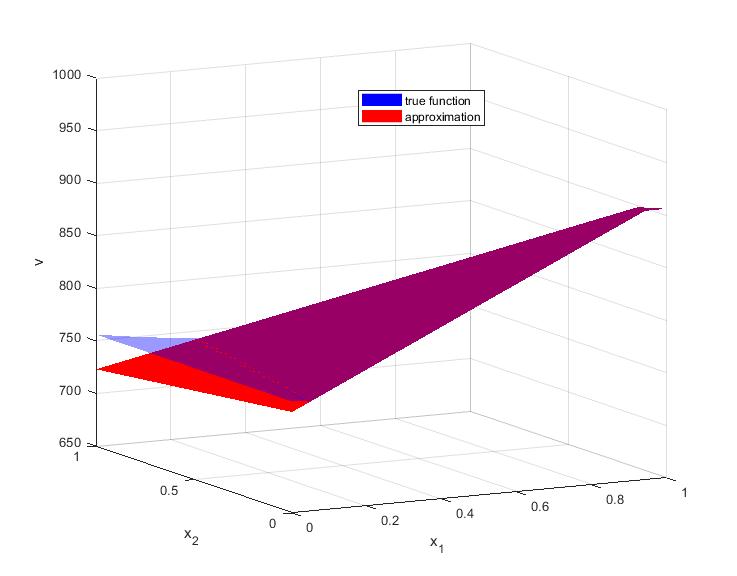}
		\captionof{figure}{The true value function $v_T$ and its concave approximation $\tilde{v}_T$}
		\label{numt_0}
	\end{minipage}
	
\end{table}

\begin{figure}[h]
	\centering
	\begin{minipage}{0.46\textwidth}
		\centering
		\includegraphics[width=\textwidth]{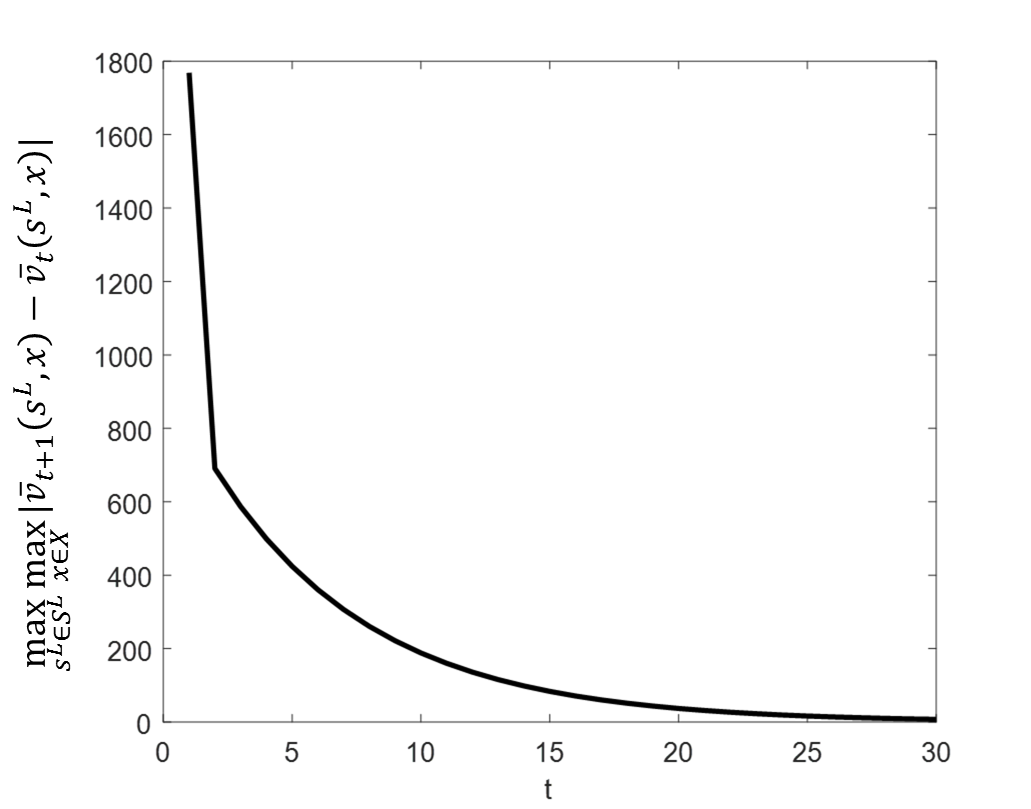}\\	
		(a) The maximum deviation of the value function $\bar{v}_{t+1}^L$ from $\bar{v}^L_t$ decreases to zero as $t$ increases.
	\end{minipage}
	\begin{minipage}{0.46\textwidth}
		\centering
		\includegraphics[width=\textwidth]{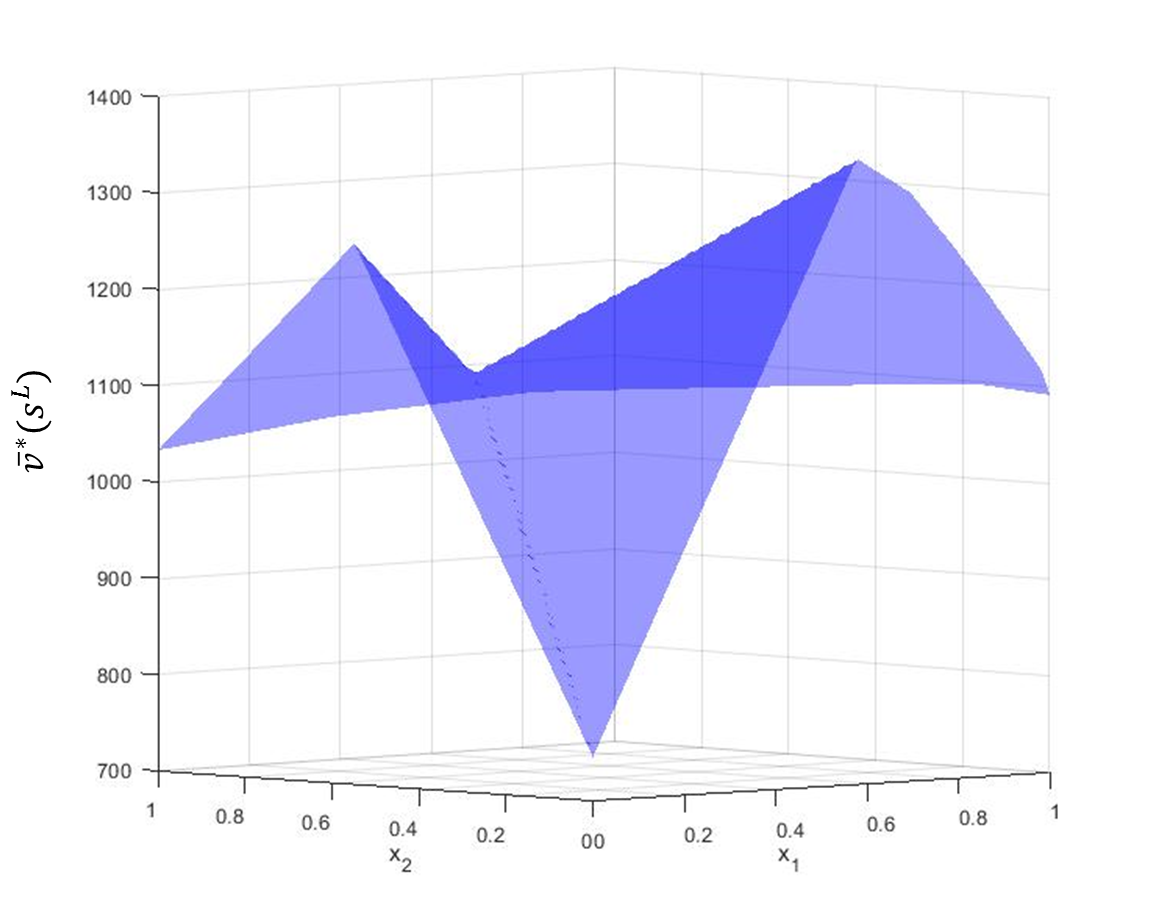}\\
		(b) The (projected) value function $\bar{v}^{L,*}$ on non-absorbing states.
	\end{minipage}\hfil
	\caption{The convergence result of the proposed algorithm.}
	\label{numt}
\end{figure}

The entire solution procedure was performed on an Intel 3.10 GHz processor having 6.00 GB memory. The total computation time for $T=30$ was 156.31 seconds, where the PURGE, DOMINANCE, and APPROXIMATION operations accounted for 4.78\%, 11.86\%, and 83.45\%, respectively. As at  least a witness point was associated with each $\gamma$-vector in the concave approximation MIP \eqref{ApproximationMIP}, the sizes of MIPs in the APPROXIMATION step are significantly larger than those of the MIPs in the PURGE step and DOMINANCE step (could be 10$\sim$20 times larger). We remark that the dynamic programming approach for POSGs in Hansen et al.(2004) quickly ran out of memory after horizon $T=4$ for a small two-agent problem (the sizes of state, observation, and action spaces for each agent are all two). 

To illustrate our results and support the validation of the solution procedure and computational results, we consider the following three scenarios:
\begin{enumerate}[(i)]
	\item Scenario 1: The policy for the leader is constructed according to Algorithm \ref{overall}, and the follower selects the action that minimizes the leader's criterion value;
	\item Scenario 2: The policy for the leader is constructed according to Algorithm \ref{overall}, and the follower randomly selects its action;
	\item Scenario 3: The leader protects the most vulnerable target (which is the target that generates the largest negative consequence if attacked), and the follower selects the action that minimizes the leader's criterion value.	
\end{enumerate}

We assume the measure of performance is the leader's total discounted reward for each sample path and present in Fig. \ref{BoxplotFig} the distribution of the performance measure for each scenario, based on 1000 simulations.  We note the following from Fig. \ref{BoxplotFig}:

\begin{enumerate}[(i)]
	\item Scenario 1 has a lower mean than Scenario 2, which illustrates that if the policy for the leader is constructed according to Algorithm \ref{overall} and the follower uses policies other than selecting actions to minimize the leader's criterion value (e.g., irrational), then the performance of the leader will only improve. That is, the proposed algorithm protects the leader against all possible decision making processes of the adversary. 
	\item Scenario 1 has a higher mean than Scenario 3, which illustrates that if the leader uses policies other than the policy constructed according to Algorithm \ref{overall} and the follower selects actions that minimize the leader's criterion, the performance of the leader will degrade. It also shows that constantly protecting the most vulnerable target is not always the best option to the leader in the presence of an intelligent adversary. This is because in a dynamic environment, the adversary will learn the leader's strategy and thus attacks other targets accordingly. It further illustrates the importance of developing dynamic defensive strategies and defensible systems for protecting against intelligent and adaptive adversaries.     
\end{enumerate}

\begin{figure}[h]
	\centering
	\includegraphics[width=0.65\textwidth]{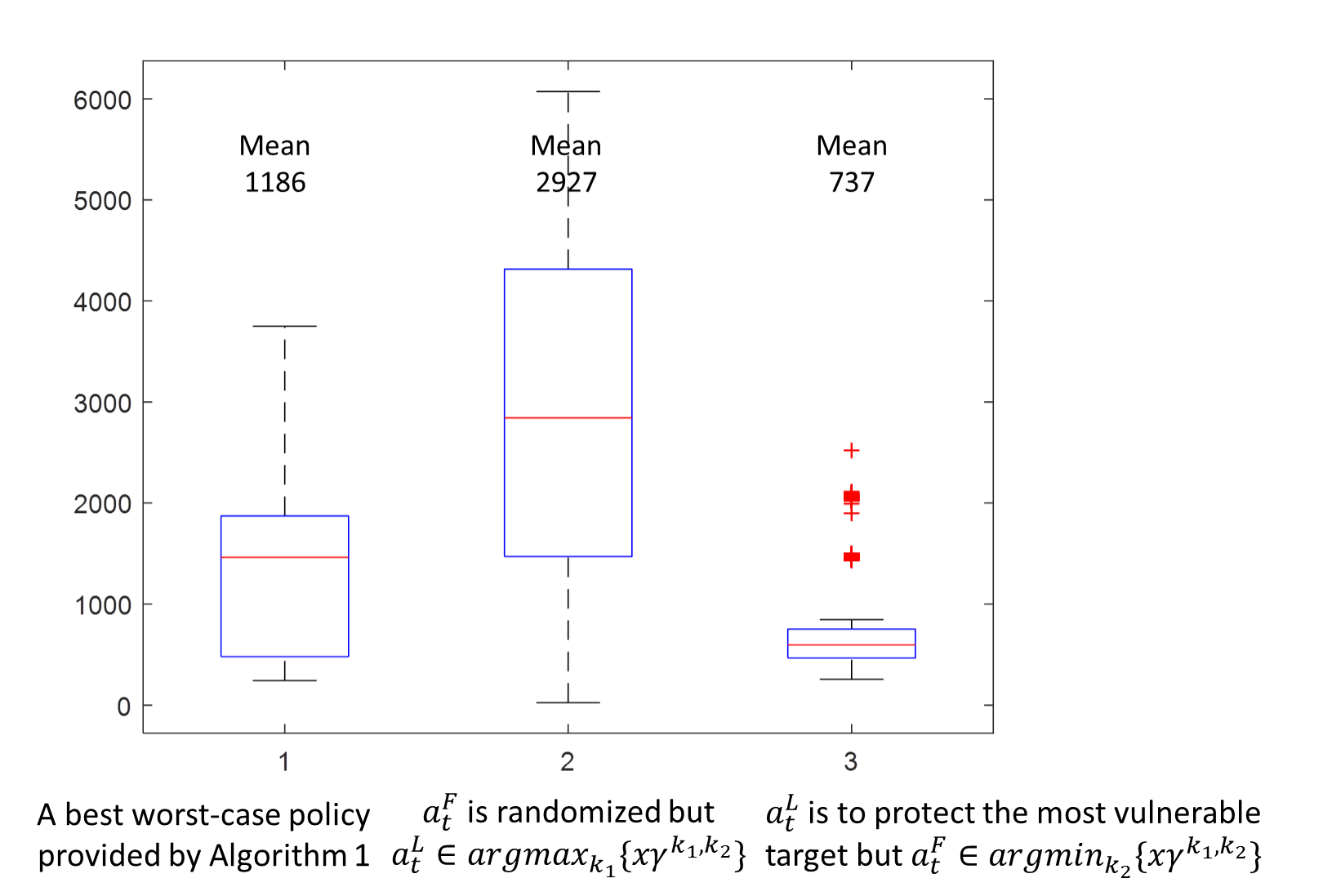}
	\caption{Comparisons among the proposed policy and two baseline policies.}
	\label{BoxplotFig}
\end{figure}
\section{Conclusions}
\label{Conclusions}
This paper introduced a worst-case analysis to a general-sum partially observable stochastic game with two non-cooperative agents, a leader and a follower, where the leader has little knowledge of the follower. A general-sum partially observable stochastic game can be transformed to a simpler single-agent problem under the worst-case scenario; however, it cannot be transformed to a standard POMDP. While the worst-case modeling can be viewed as a POMDP with imprecise parameters, there is no as yet established algorithms for these problems in the literature. In order to determine a baseline performance for the leader, we showed the optimal value function of the leader can be non-convex and developed a recursive algorithm to construct a lower bound of the leader's value function in the finite horizon and its associated policy. We further analyzed the quality of the lower bound and showed that the proposed procedure can be used to approximately evaluate the leader's performance in the infinite horizon case. The use of the model and the solution procedure was illustrated by a liquid egg production example in a security context.    

The lower bound was constructed by the sets $\tilde{\Gamma}_{t}(s^L) \subset \Gamma_{t}(s^L), \forall s^L \in S^L$. Future research should further improve the lower bound by efficiently searching $\gamma \notin \Gamma_{t}(s^L)$. Another research direction would be a detailed discussion for the infinite planning horizon problem. Moreover, the developed analysis provided a benchmark result for an agent in a dynamic, multi-agent partially observable stochastic environment. Thus, the follow-up research may include investigating the value of improved understanding of the adversarial behaviors by comparing with these benchmark results.

%
%
%




\begin{thebibliography}{9}
	
	\bibitem{Aghassi}
	Aghassi M, Bertsimas D (2006), Robust game theory. \emph{Math. Program.} 107(1), 231-273. 
	
	\bibitem{Ahmadi}
	Ahmadi M, Cubuktepe M, Jansen N, Junges S, Katoen J, Topcu U (2018), The partially observable games we play for cyber deception. Retrieved from https://arxiv.org/pdf/1810.00092.pdf .
	
	\bibitem{An}
	An B, Shieh E, Yang R, Tambe M, Baldwin C, DiRenzo J, Maule B, Meyer G (2014), Protect-a deployed game theoretic system for strategic security allocation for the United States coast guard. \emph{AI Mag.} 33(4), 96-110.  
	
	\bibitem{Bakir}
	Bakir NO, Kardes E (2009), A stochastic game model on overseas cargo container security. Technical Report, May 3.
	
	\bibitem{Berg}
	De Berg M, Cheong O, van Kreveld M, Overmars M (2008), Computational Geometry: Algorithms and Applications. Springer-Verlag, Berlin. 
	
	\bibitem{Bernstein}	
	Bernstein DS, Givan R, Immerman N, Zilberstein S (2002), The complexity of decentralized control of Markov decision processes. \emph{Math. Oper. Res.} 27(4), 819-840.  
	
	\bibitem{Bier}
	Bier VM, Oliveros S, Samuelson L (2007), Choosing what to protect. \emph{J. Public Econ. Theory}. 9(4), 563-587. 
	
	\bibitem{Camerer2011}
	Camerer C (2011), Behavioral game theory: Experiments in strategic interaction. Princeton University Press.
	
	\bibitem{Caprara}
	Caprara A, Carvalho M, Lodi A, Woeginger GJ(2016), Bilevel knapsack with interdiction constraints. \emph{INFORMS J. COMPUT.} 28(2), 319-333. 
	
	\bibitem{Cassandra}
	Cassandra AR, Littman ML, Zhang NL (1997), Incremental pruning: A simple, fast, exact algorithm for partially observable Markov decision processes. \emph{Proc. Thirteenth Ann. Conf. on Uncertainty in Artificial Intelligence}, Providence, RI, 54-61.
	
	\bibitem{Chang2015a}
	Chang Y, Erera AL, White CC (2015a), A leader-follower partially observed Markov game. \emph{Ann. Oper. Res}. 235(1), 103-128. 
	
	\bibitem{Chang2015b}
	Chang Y, Erera AL, White CC (2015b), Value of information for a leader-follower partially observed Markov game. \emph{Ann. Oper. Res.} 235(1), 129-153. 
	
	
	\bibitem{Chang2019}
	Chang Y, Keblis MF, Li R, Iakovou E, White CC (2019), Value of Misinformation and Disinformation in Modern Warfare. \emph{Oper. Res}. under review (second round). 
	
	\bibitem{Cheng}
	Cheng HT (1988), Algorithms for partially observable Markov decision processes. \emph{Ph.D. thesis}, University of British Columbia, British Columbia, Canada.
	
	\bibitem{Cheng16}	
	Cheng J, Leung J, Lisser A (2016), Random-payoff two-person zero-sum game with joint chance constraints. \emph{Eur. J. Oper. Res.} 252(1), 213-219. 
	
	\bibitem{Conlisk}
	Conlisk J (1996), Why bounded rationality? \emph{J. Econ. Lit.} 669-700. 
	
	\bibitem{DHS}
	Department of Homeland Security (2015), Moving target defense. Available: https://www.dhs.gov/science-and-technology/csd-mtd.
	
	\bibitem{DoshiandGmytraslewicz}
	Doshi P, Gmytraslewicz P (2009), Monte Carlo sampling methods for approximating interactive POMDPs. \emph{J. Artif. Intell. Res.} 34, 297-337. 
	
	\bibitem{DoshiandPerez}
	Doshi P, Perez D (2008), Generalized point based value iteration for interactive POMDPs. \emph{23rd Conf. on Artif. Intell.} 63-68. 
	
	\bibitem{Emery-Montemerlo}	
	Emery-Montemerlo R, Gordon G, Schneider J, Thrun S (2004), Approximate solutions for partially observable stochastic games with common payoffs. \emph{Proc. Third Int. Joint Conf. on Auton. Agent Multi Agent Syst (AAMAS)}, 136-143. 
	
	\bibitem{Ghosh}
	Ghosh MK, McDonald D, Sinha S (2004), Zero-sum stochastic games with partial information. \emph{J. Optimiz. Theory App.} 121(1), 99-118. 
	
	\bibitem{Gilboa}
	Gilboa I, Schmeidler D (1989), Maxmin expected utility with non-unique prior. \emph{J. Math. Econ.} 18(2), 141-153.  
	
	\bibitem{Gmytrasiewicz}
	Gmytrasiewicz PJ, Doshi P (2004), Interactive POMDPs: properties and preliminary results. \emph{Proc. Third Int. Joint Conf. on Auton Agent Multi Agent Syst (AAMAS)}. 3, 1374-1375. 
	
	\bibitem{Goldsmith}
	Goldsmith J, Mundhenk M (2008), Competition adds complexity. In J.C. Platt, D. Koller, Y. Singer, and S.Roweis, editors, Advances in Neural Information Processing Systems 20, 561-568, MIT Press, Cambridge, MA.
	
	\bibitem{Hansen}
	Hansen EA, Bernstein D, Zilberstein S (2004), Dynamic programming for partially observable stochastic games. \emph{Proc. Nineteenth National Conf. on Artificial Intelligence}. San Jose, California, 709-715.
	
	\bibitem{Haskell}
	Haskell W, Kar D, Fang F, Tambe M, Cheung S, Denicola E (2014), Robust protection of fisheries with compass. \emph{Proc. 26th Annu. Conf. Innov. Appl. Artif. Intell.} 2978-2983. 
	
	\bibitem{Itoh}
	Itoh H, Nakamura K (2007), Partially observable Markov decision processes with imprecise parameters. \emph{Artif. Intell.} 171, 453-490. 
	
	\bibitem{Jain}
	Jain M, Tsai J, Pita J, Kiekintveld C, Rathi R, Tambe M (2010), Software assistants for randomized patrol planning for the LAX airport police and the Federal Air Marshal Service. \emph{Interfaces}. 40(4), 267-290. 
	
	\bibitem{Kardes2011}
	Kardes E, Ordonez F, Hall RW (2011), Discounted robust stochastic games and an application to queueing control. \emph{Oper. Res}. 59(2), 365-382. 
	
	\bibitem{Kardes2014}	
	Kardes E (2014), On discounted stochastic games with incomplete information on payoffs and a security application. \emph{Oper. Res. Lett.} 42, 7-11. 
	
	\bibitem{Lin}
	Lin Z, Bean JC, White CC (2004), A hybrid genetic/optimization algorithm for finite-horizon, partially observed Markov decision processes. \emph{INFORMS J. COMPUT.} 16(1), 27-38. 
	
	\bibitem{Lo}
	Lo, KC (1996), Equilibrium in beliefs under uncertainty. \emph{J. Econ. Theory}. 71(2), 443-484. 
	
	\bibitem{Lopez}
	Lopez D, Pastor O, Villalba LG (2013), Dynamic risk assessment in information system: state-of-the-art. \emph{Proc. 6th Int. Conf. on Infor. Tech.} Amman, 8-10.  
	
	\bibitem{Luque-Vasquez}
	Luque-Vasquez F, Minjarez-Sosa JA (2013), Average optimal strategies for zero-sum Markov games with poorly known payoff function on one side. \emph{J. Dynam. Game}. 1(1), 105-119. 
	
	\bibitem{March}
	March J (1978), Bounded rationality, ambiguity, and the engineering of choice. \emph{Bell J. Econ.}, 587-608. 
	
	\bibitem{Marinacci}
	Marinacci M (2000) Ambiguous games. \emph{Games Econ. Behav.} 31(2), 191-219.
	
	\bibitem{Miehling}
	Miehling E, Rasouli M, Teneketzis D (2015), Optimal defense policies for partially observable spreading processes on Bayesian attack graphs. \emph{Proc. Second ACM Workshop on Moving Target Defense}, 6-76.  
	
	\bibitem{Minjarez-Sosa}
	Minjarez-Sosa JA, Vega-Amaya O (2009), Asymptotically optimal strategies for adaptive zero-sum discounted Markov games. \emph{SIAM J. Control Optim.} 48(3), 1405-1421.
	
	\bibitem{Najim}
	Najim K, Poznyak AS, Gomez E (2001), Adaptive policy for two finite Markov chains zero-sum stochastic game with unknown transition matrices and average payoffs. \emph{Automatica}. 37(7), 1007-1018.
	
	\bibitem{Nguyen}
	Nguyen TH, Yang R, Azaria A, Kraus S, Tambe M (2013), Analyzing the effectiveness of adversary modeling in security games. \emph{Proc. Twenty-Seventh AAAI Conf. Artificial Intelligence}. 718-724.
	
	\bibitem{Oliehoek2012}
	Oliehoek FA (2012) Decentralized POMDPs. \emph{Reinforcement Learning: State of the Art, Adaptation, Learning, and Optimization}, Springer Berlin Heidelberg.
	
	\bibitem{Oliehoek2016}
	Oliehoek FA, Amato C (2016), Infinite-horizon DEC-POMDPs. \emph{A Concise Introduction to Decentralized POMDPs}, pp. 69-77. Springer.
	
	\bibitem{Osogami}
	Osogami T(2015), Robust partially observable Markov decision process. \emph{Proc. 32nd Int. Conf. on Machine Learning}. 37.
	
	
	\bibitem{Pineau}
	Pineau J, Gordon G, Thrun S (2003) Point-based value iteration: An anytime algorithm for POMDPs. \emph{Proc. Int. Joint Conf. on
		Artificial Intelligence}, Acapulco, Mexico. 1025-1030.
	
	\bibitem{Poolsappasit}
	Poolsappasit N, Dewri R, Ray I (2012) Dynamic security risk management using Bayesian attack graphs. \emph{IEEE Trans. Dependable Secure Comput.} 9(1), 61-74.  
	
	\bibitem{Rabinovich}
	Rabinovich Z, Goldman CV, Rosenschein JS (2012), The complexity of multiagent systems: The price of silence. \emph{Proc. Int. Conf. on
		Auton. Agent Multi Agent Syst (AAMAS)}. Melbourne, Australia, 1102-1103.
	
	\bibitem{Rasouli}
	Rasouli M, Saghafian S (2018), Robust Partially Observable Markov Decision Processes. \emph{Submitted}.
	
	\bibitem{Rosenberg}
	Rosenberg D, Solan E, Vieille N (2004), Stochastic games with a single controller and incomplete information. \emph{SIAM J. Control Optim.} 43(1), 86-110.
	
	\bibitem{Saghafian}	
	Saghafian S (2018), Ambiguous partially observable Markov decision processes: Structural results and applications, \emph{J. Econ. Theory}. 178, 1-35.
	
	\bibitem{Saha}
	Saha S (2014), Zero-sum stochastic games with partial information and average payoff. \emph{J. Optimiz. Theory App.} 160(1), 344-354.
	
	\bibitem{Seuken}
	Seuken S, Zilberstein S (2005), Formal models and algorithms for decentralized control of multiple agents, \emph{Technical Report}, 05-68, Department of Computer Science, University of Massachusetts, Amherst, MA 01003.
	
	\bibitem{Shameli-Sendi}
	Shameli-Sendi A, Ezzati-Jivan N, Jabbarifar M, Dagenais M (2012), Intrusion response systems: survey and taxonomy. \emph{Int. J. Comput. Sci. Netw.} 12(1), 1-14. 
	
	\bibitem{Shani}
	Shani G, Pineau J, Kaplow R (2013), A survey of point-based POMDP solvers. \emph{Auton. Agents Multi Agent Syst.} 27(1), 1-51.
	
	\bibitem{Shapley}
	Shapley LS(1953), Stochastic games. \emph{Proc. Natl. Acad. Sci. USA.} 39(10), 1095-1100.
	
	\bibitem{Simchi}
	Simchi-Levi D, Wei Y (2015), Worst-case analysis of process flexibility designs. \emph{Oper. Res}. 63(1), 166-185.
	
	\bibitem{Sondik}
	Sondik EJ (1971), The optimal control of partially
	observable Markov processes. \emph{Ph.D. thesis}, Stanford University.
	
	\bibitem{Sonu}	
	Sonu E, Doshi P (2012), Generalized and bounded policy iteration for finitely
	nested interactive POMDPs: Scaling up. \emph{Proc. Twelfth Int. Conf. on
		Auton. Agent Multi Agent Syst (AAMAS)}, 1039-1048.
	
	\bibitem{Sorina}
	Sorin S (1984), ``Big Match" with lack of information on one side (part i). \emph{Internat. J. Game Theory.} 13(4), 201-255.
	
	\bibitem{Sorinb}	
	Sorin S (1985), ``Big Match" with lack of information on one side (part ii). \emph{Internat. J. Game Theory.} 14(3), 173-204.
	
	
	\bibitem{USDA2015}
	United States Department of Agriculture Food Safety and Inspection Service (2015), Food Safety Information: Egg Products and Food Safety.
	
	\bibitem{Joint2012}
	U.S. Joint Chiefs of Staff (2012), Information Operations. Joint Publication 3-13. Washington, DC, 27 November, Incorporating Change 1, 20 November 2014. 
	
	\bibitem{Wang2019}
	Wang X, Tambe M, Bosansky B, An B (2019), When players affect target values: modeling and solving dynamic partially observable security games. \emph{OptMAS workshop @ AAMAS}.
	
	\bibitem{Wiggers2016}
	Wiggers AJ, Oliehoek FA, Roijers DM (2016), Structure in the value function of two-player zero-sum games of incomplete information.  \emph{Proc. Twenty-second Eur. Conf. on Artificial Intelligence}. The Hague, Netherlands, 1628-1629. 
	
	\bibitem{Yang}
	Yang R, Ford B, Tambe M, Lemieux A (2014), Adaptive resource allocation for wildlife protection against illegal poachers. \emph{Int. Conf. on Auton. Agents and Multi-Agent Syst.} 453-460. 
	
	\bibitem{Yolmeh}
	Yolmeh A, Baykal-Gursoy M (2017), A robust approach to infrastructure security games. \emph{Comput. Ind. Eng.} 110, 515-526.
	
	\bibitem{Zhang2010}
	Zhang H (2010), Partially observable Markov decision processes: A geometric technique and analysis. \emph{Oper. Res.} 58, 214-228.
	
	\bibitem{Zhang2013}
	Zhang Y (2013), Contributions in supply chain risk assessment and mitigation. \emph{Ph.D. thesis}, Georgia Institute of Technology.
	
	\bibitem{Zhuang2014}
	Zhuang R, Deloach SA, Ou X (2014), Towards a theory of Moving Target Defense. \emph{Proc. First ACM Workshop on Moving Target Defense}. 31-40.
\end{thebibliography}



\bibliographystyle{nonumber}

\end{document}